\documentclass[11pt]{article}
\usepackage[a4paper,margin=25mm]{geometry}
\usepackage{amsmath,amssymb,amsfonts,amsthm}
\usepackage{mathrsfs}
\usepackage{mathtools}
\usepackage{stmaryrd}
\usepackage{bbold}
\usepackage{breqn}
\usepackage{microtype}
\usepackage{txfonts}
\usepackage{geometry}
\usepackage[colorlinks,allcolors=blue]{hyperref}

\usepackage[mathscr]{euscript}
\usepackage{booktabs}
\usepackage{multirow} 
\usepackage{graphicx}
\usepackage{comment}
\makeatletter
\newcommand*{\defeq}{\mathrel{\rlap{%
					 \raisebox{0.3ex}{$\m@th\cdot$}}%
					 \raisebox{-0.3ex}{$\m@th\cdot$}}%
					 =}
\makeatother
\makeatletter
\newcommand*{\eqdef}{=\mathrel{\rlap{%
					 \raisebox{0.3ex}{$\m@th\cdot$}}%
					 \raisebox{-0.3ex}{$\m@th\cdot$}}%
					 }
\makeatother

\makeatletter
\newcommand*\bcdot{\mathpalette\bcdot@{.5}}
\newcommand*\bcdot@[2]{\mathbin{\vcenter{\hbox{\scalebox{#2}{$\m@th#1\bullet$}}}}}
\makeatother

\theoremstyle{definition}
\newtheorem{definition}{Definition}[section]

\theoremstyle{plain}

\newtheorem{corollary}[definition]{Corollary}
\newtheorem{theorem}[definition]{Theorem}
\newtheorem{proposition}[definition]{Proposition}

\theoremstyle{remark}

\theoremstyle{definition}
\newtheorem{example}[definition]{Example}

\usepackage{authblk}
\usepackage{cleveref}
\usepackage{etoolbox}
\makeatletter
\patchcmd{\cref@old@eq@setnumber}{\stepcounter{equation}}{\refstepcounter{equation}}{}%
 {\PackageError{auf-revision}{Unable to patch breqn equation references}{}}
\makeatother

\DeclarePairedDelimiter\abs{\lvert}{\rvert}%
\DeclarePairedDelimiter\norm{\lVert}{\rVert}%
\DeclarePairedDelimiter\croc{\langle}{\rangle}%

\makeatletter
\let\oldabs\abs
\def\abs{\@ifstar{\oldabs}{\oldabs*}}
\let\oldnorm\norm
\def\norm{\@ifstar{\oldnorm}{\oldnorm*}}
\let\oldcroc\croc
\def\croc{\@ifstar{\oldcroc}{\oldcroc*}}
\makeatother

\makeatletter
\newcommand{\proofstep}[1]{%
  \par
  \addvspace{\medskipamount}%
  \noindent\textit{#1\@addpunct{.}}\enspace\ignorespaces
}
\makeatother

\newcommand{\RE}{\mathop{\mathrm{Re}}}
\newcommand{\IM}{\mathop{\mathrm{Im}}}
\newcommand{\Res}{\operatorname{Res}}
\newcommand{\Log}{\operatorname{Log}}
\newcommand{\Arg}{\operatorname{Arg}}

\newcommand{\ee}{\operatorname{e}}
\newcommand{\uo}{\mathfrak{u}}
\newcommand{\dd}{\operatorname{d}\!}

\newcommand{\jj}{\mathrm j}

\crefname{definition}{Definition}{Definitions}
\crefname{theorem}{Theorem}{Theorems}
\crefname{proposition}{Proposition}{Propositions}
\crefname{corollary}{Corollary}{Corollaries}
\crefname{lemma}{Lemma}{Lemmas}

\DeclareMathAlphabet{\mathcal}{OMS}{cmsy}{m}{n}
\SetMathAlphabet{\mathcal}{bold}{OMS}{cmsy}{b}{n}

\hypersetup{pdftitle={A Resurgent Analytic Framework for Indicial Umbral Calculus},pdfauthor={Roberto Ricci}}
\begin{document}
\title{A Resurgent Analytic Framework for Indicial Umbral Calculus via Mellin--Barnes and Borel--Laplace Theories}
\author{Roberto Ricci\footnote{E-mail: roberto.ricci@enea.it (Roberto Ricci)}}
\affil[1]{ENEA, Nuclear Department NUC-DTT, Frascati Research Center, Via E. Fermi 45, 00044 Frascati (Rome), Italy}
\date{}
\setcounter{Maxaffil}{0}
\renewcommand\Affilfont{\itshape\small}
\maketitle

\begin{abstract}
We introduce the Analytic Umbral Framework (AUF), which gives formal
indicial umbral evaluation an analytic meaning through Mellin--Barnes
pairings. An auxiliary Borel representative defines a continuous
linear action on compatible weighted spaces of analytic ground states.
The resulting function is initially defined in a sector of the
logarithmic covering. Finite contour displacement recovers the formal
umbral coefficients, with remainder estimates distinguishing ordinary
asymptotics, strong Gevrey asymptotics and equality with a convergent
series. The associated auxiliary Laplace series are resurgent for
rational profiles and entire profiles of exponential type. Under
explicit spectral conditions, the pairing coincides with the
Borel--Laplace sum of the evaluated formal series. An exact spectral
transmutation relates exponential and rational representatives and
extends to admissible P{\'o}lya superpositions. It preserves the
analytic value under a compensating change of ground state. A Hankel
construction supplies complementary realizations for inverse-Gamma
states. Worked examples involving special functions illustrate
convergent and divergent evaluations, logarithmic expansions and the
justification of operational integration. In the zero-dimensional
quartic model, a single pairing yields the weak- and strong-coupling
expansions and an exact nonperturbative lateral connection.
\end{abstract}


\section{Introduction}\label{sec:introduction}

Umbral calculus replaces numerical sequences with symbolic expressions evaluated on suitable ground states, and makes it possible to simplify calculations by manipulating those symbols before evaluation. Its origins lie in the nineteenth-century symbolic calculus \cite{JordanBell1965,Bell1938, aDB00}. The classical theory of Rota and Roman gives this procedure a precise algebraic meaning through linear functionals on polynomial spaces \cite{RomanRota1973,sR84}. The approach considered here is \emph{indicial umbral calculus}, which draws on poweroid and operational methods \cite{Steffensen1941,Dattoli2000} and has been developed for calculations involving special functions and integral transforms \cite{sL22,Babusci2019,gD24}.

In this approach, a function is represented by a simpler expression containing an auxiliary symbol $\uo$. Powers of that symbol act on a \emph{ground state function} through the rule $\uo^n[\varphi]=\varphi(n)$. Thus, for $F(w)=\sum_{n=0}^{\infty}a_nw^n$, the basic prescription is
\begin{dmath}\label{eq:indicial-rule}
 \widehat{\mathcal A}(\zeta)
 \defeq {F(\zeta\uo)[\varphi]
 =\sum_{n=0}^{\infty}a_n\varphi(n)\zeta^n}
\end{dmath}.
One performs algebraic or differential manipulations on the auxiliary expression and evaluates its powers only afterwards. This can simplify calculations considerably, but the right-hand side of \eqref{eq:indicial-rule} is initially a formal series. Its convergence and its relation to an analytic function require separate consideration.

Earlier work has connected these operational methods with the integral Borel transform \cite{gD15,gD20}, and a recent formal-series formulation is given in \cite{Ricci2026}. The question addressed here concerns the analytic meaning of the evaluation. A divergent series may describe the asymptotic behaviour of a well-defined function. Conversely, convergence of a series does not by itself justify integrating it term by term over an unbounded interval. An analytic formulation must therefore specify the evaluated function, its domain and the conditions under which the intended operations are legitimate.

We introduce the \emph{Analytic Umbral Framework} (AUF) to address this problem. The construction retains the ground state function $\varphi$, but replaces the symbolic umbral expression with an \emph{umbral Borel functional representative}, that is an ordinary function $\Delta(\ee^u;\zeta)$ of the auxiliary complex variable $u$. For the expression in \cref{eq:indicial-rule}, this is simply $\Delta(\ee^u;\zeta)=F(\zeta \ee^u)$. The same function is written $\Delta(z;\zeta)$ when using the multiplicative coordinate $z=\ee^u$. On a suitable auxiliary ray and in a strip of convergence, its Mellin transform in $z$ provides a spectral representative, called the \emph{jump function} and denoted by $\jj(t;\zeta)$ \cite{Titchmarsh1948,Zemanian1987}. The analytic evaluation is then constructed through a Mellin--Barnes pairing \cite{ParisKaminski2001},
\begin{dmath}\label{eq:intro-pairing}
 \widetilde{\mathcal A}(r,\theta)
 \defeq {\langle\Delta,\varphi\rangle_{\mathcal C}(r,\theta)
 =\frac{1}{2\pi i}\int_{\mathcal C}
 \widetilde \jj(t;r,\theta)\,\varphi(t)\dd t}
\end{dmath},
where $\mathcal C$ is an admissible spectral contour. Here $(r,\theta)$ records the modulus and a continuously chosen argument of the external parameter $\zeta$. Complex powers in the integrand make this argument relevant, so the pairing generally converges first in a sector of the logarithmic covering. A planar determination is obtained by restricting the argument to a chosen interval of length $2\pi$, wherever the lifted function is defined.

The name \emph{umbral Borel functional} refers to the way the representative $\Delta$ acts linearly on ground states, rather than to a different kind of scalar function. We establish continuity of this action on compatible weighted spaces of ground states. This follows the test-function and continuous-dual viewpoint of generalised-function theory \cite{GelfandShilov1968}; the required spaces and estimates are specified as part of the construction.

The pairing also explains how the formal rule is recovered. Under stated analyticity and growth conditions, displacement of the spectral contour gives a finite residue sum and an integral remainder. The residues yield the coefficients in \eqref{eq:indicial-rule}, while estimates of the remainder determine whether the series is asymptotic or converges to the analytic value. Thus the analytic interpretation of the formal calculation follows from the pairing, rather than being assumed in its definition. The ground state, contour and logarithmic determination supply analytic information beyond the coefficient sequence.

A second main result is a \emph{spectral transmutation law}. The jump function of the elementary rational representative is that of the exponential representative multiplied by $\Gamma(1+t)$. Dividing the accompanying ground state by the same factor gives an equal pairing on a common admissible contour. We extend this relation to suitable exponential superpositions and explain how the corresponding formal identities follow. For divergent expansions, we also give conditions for agreement with Borel--Laplace summation and examine lateral continuation in relation to Stokes phenomena \cite{Balser2000,Costin2009,dS16,mLR16}. The framework starts from analytically specified representatives and ground states; it does not require every  formal series resulting from contour displacement to be Borel summable. The examples also provide an analytic realization of an expansion outside every finite Gevrey class. In that setting, the analytic state and contour remain part of the supplied data.

The reported examples distinguish convergent evaluations, divergent asymptotic expansions and cases requiring a Hankel rather than a vertical-contour realisation. The zero-dimensional quartic model provides a sustained comparison with established nonperturbative results \cite{MalbouissonPortugalSvaiter2001,Marucho2008,Rivasseau2009}. It brings together the weak- and strong-coupling expansions, Borel summation and the exact lateral connection within one analytic construction.

The article is organised as follows. \Cref{sec:preliminaries} fixes the transform and asymptotic conventions. \Cref{sec:umbral_borel_functionals} constructs the Borel representatives and their jump functions, and \cref{sec:spectral_duality} establishes spectral duality. The analytic pairing, its domain and the recovery of formal evaluation are developed in \cref{sec:umbral_pairing}. \Cref{sec:examples} presents the worked examples, followed by the quartic model in \cref{sec:numerical-illustration}. \Cref{sec:conclusion} summarises the results and their scope. Technical proofs and estimates are collected in the appendices.

\subsection*{Notations and conventions}

The variables $u$, $z=\ee^u$, $t$ and $\zeta$ denote, respectively, the additive auxiliary coordinate, the multiplicative auxiliary coordinate, the spectral variable and the external parameter. We use $\Delta$ for a generic Borel representative and separate symbols for explicit examples. The jump function is denoted by $\jj$, and $\varphi$ denotes the ground state. 

A hat denotes a formal series, whether convergent or divergent. In particular, $\widehat\Delta(t;\zeta)$ denotes a formal inverse-power series in $t$, whereas $\widehat{\mathcal A}(\zeta)$ denotes a formal expansion of an evaluated pairing in $\zeta$. Analytic Borel representatives carry no accent; their relation to formal series is indicated by the Borel-transform operator $\mathcal B$.

A tilde denotes a function on the logarithmic covering. We write its points as $(r,\theta)$, with $r>0$ and $\theta\in\mathbb R$, and their planar projection as $\zeta=r\ee^{i\theta}$. Our planar logarithm convention is
\begin{dmath}\label{eq:log-convention}
 \log_\phi\zeta
 \defeq {\ln|\zeta|+i\arg_\phi\zeta,
 \qquad \phi-2\pi<\arg_\phi\zeta<\phi}
\end{dmath},
with $\Log\equiv\log_\pi$. Accordingly,
\begin{dmath}\label{eq:intro-planar-determination}
 \mathcal A_\phi(\zeta)
 \defeq \widetilde{\mathcal A}\bigl(|\zeta|,\arg_\phi\zeta\bigr)
\end{dmath}
where the corresponding lifted point belongs to the domain of $\widetilde{\mathcal A}$. Complex powers use the specified logarithmic determination. Vertical spectral contours are oriented upward. We take $\mathbb N=\{0,1,2,\ldots\}$.

\section{Preliminaries}\label{sec:preliminaries}

We recall the formal-series operations and analytic transform conventions used below. For summability theory see \cite{Balser1994,Balser2000,dS16,mLR16}; for Mellin transforms and contour integrals see \cite{Titchmarsh1948,ParisKaminski2001}.

\subsection[The multiplicative algebra of formal inverse-power series]{The multiplicative algebra $\mathbb C[[t^{-1}]]$}\label{sec:multiplicative_algebra}

The space $\mathbb C[[t^{-1}]]$ consists of formal series in nonpositive powers of $t$, with coefficientwise addition and the Cauchy product \cite{cC15}. Formal differentiation with respect to $t$ obeys the usual product rule. We shall mainly use its nonunital subalgebra
\begin{dmath}\label{eq:formal-Laplace-space}
 t^{-1}\mathbb C[[t^{-1}]]
 ={\left\{\widehat f(t)=\sum_{n=0}^{\infty}c_n t^{-n-1}\,\middle|\,c_n\in\mathbb C\right\}}
\end{dmath},
which we call the \emph{formal Laplace space}. A constant term, when present, is treated separately.

In the formal topology, convergence means eventual stabilization of each coefficient \cite{dB01,dS16}. The partial sums of every formal series therefore converge formally, regardless of whether substitution of a complex value for $t$ gives a convergent numerical series.

\subsection{Formal Borel transform}\label{sec:formal_Borel_transform}

The Borel transform compensates for factorial growth of the coefficients. Our convention is, for $\widehat f \in t^{-1}\mathbb C[[t^{-1}]]$, 
\begin{dmath}\label{eq:Borel_transform}
 (\mathcal B\widehat f\;)(u)
 \defeq {\sum_{n=0}^{\infty}\frac{c_n}{n!}u^n,
 \qquad \mathcal B^{-1}[u^n]=n!t^{-n-1}}
\end{dmath}.
This linear isomorphism maps the formal Laplace space to $\mathbb C[[u]]$. When the output converges, we denote its analytic germ by $b(u)$, without a hat.

Two elementary criteria describe the passage to analytic functions. The Borel series converges near $u=0$ exactly when $|c_n|\leq CA^n n!$ for some $C,A>0$. This is the Gevrey-one coefficient condition. Moreover, $\widehat f$ converges for sufficiently large $|t|$ exactly when its Borel transform extends to an entire function of finite exponential type, that is, $|b(u)|\leq C\ee^{A|u|}$ throughout the plane, with suitable constants. Both criteria follow from coefficient bounds and Cauchy's inequalities \cite{dS16}. An entire Borel transform can therefore correspond to a divergent Laplace series when it has faster growth in some directions.

\subsection[The convolutive algebra of formal power series]{The convolutive algebra $\mathbb C[[u]]$}\label{sec:convolutive_algebra}

The product corresponding to multiplication in the Laplace space is convolution. For analytic germs it is
\begin{dmath}\label{eq:Borel-convolution}
 (b_1*b_2)(u)\defeq{\int_0^u b_1(v)\,b_2(u-v)\dd v,
 \qquad \mathcal B(\widehat f_1\widehat f_2)
 =(\mathcal B\widehat f_1)*(\mathcal B\widehat f_2)}
\end{dmath}.
The integral is initially taken along the straight segment in a common disc of analyticity. Its coefficientwise version defines convolution for formal series as well: the convolution of $u^m/m!$ and $u^n/n!$ is $u^{m+n+1}/(m+n+1)!$. Thus the Borel transform is also an algebra isomorphism for these two products \cite{dS16}. Neither algebra has a unit within the spaces just specified. In particular, the constant Borel function $1$ is not a convolution unit.

\subsection{Laplace transform}\label{sec:Laplace_transform}

Let $b$ be holomorphic near the origin and admit analytic continuation along a neighbourhood of the ray $\ee^{i\vartheta}\mathbb R_{\geq0}$. If $|b(r\ee^{i\vartheta})|\leq C\ee^{ar}$ for $r\geq0$, the directional Laplace transform
\begin{dmath}\label{eq:directional_Laplace_transform}
 (\mathcal L^\vartheta b)(t)
 \defeq {\int_0^{\infty\ee^{i\vartheta}}\ee^{-tu}b(u)\dd u,
 \qquad \RE(t\ee^{i\vartheta})>a}
\end{dmath}
converges absolutely and locally uniformly. For $b(u)=u^n/n!$ it gives $t^{-n-1}$, which explains the normalisation in \eqref{eq:Borel_transform}. If $b$ continues across a sector and satisfies uniform exponential bounds on its closed angular subintervals, the integration ray can be rotated wherever the connecting arcs vanish. The resulting identities extend to connected common domains by analytic continuation. Bounds on two rays alone do not justify the rotation between them \cite{dS16,mLR16}.

When a singularity lies on a ray, one instead compares contours passing on its two sides. The lateral transforms $\mathcal L^{\vartheta^\pm}b$ are obtained by approaching from increasing and decreasing angles, whenever the corresponding limits exist. We write
\begin{dmath}\label{eq:Stokes-discontinuity}
 \operatorname{Disc}_\vartheta\mathcal L b
 \defeq \mathcal L^{\vartheta^+}b-\mathcal L^{\vartheta^-}b
\end{dmath}.
If finitely many simple poles $\omega_j$ are crossed and the remaining connecting integrals vanish, the residue theorem gives
\begin{dmath}\label{eq:stokes_simple_poles}
 (\operatorname{Disc}_\vartheta\mathcal L b)(t)
 =-2\pi i\sum_j\Res_{u=\omega_j}b(u)\,\ee^{-t\omega_j}
\end{dmath}.
The minus sign comes from the clockwise contour formed by the increasing-angle ray outward and the decreasing-angle ray inward. An infinite pole set requires an additional limiting argument. The signs $\vartheta^\pm$ specify angular approach; on the negative real axis they do not coincide with the geometrical upper/lower bank ordering.

Resurgence studies the analytic continuation of Borel germs and the structure of their singularities \cite{Ecalle1981,dS16}. In the examples below, continuation along paths avoiding a fixed locally finite singular set is sufficient; entire germs also satisfy this continuation property. We use resurgence in this continuation sense and impose the growth conditions for Laplace summation separately. This distinction matters because the absence of finite singularities does not ensure moderate growth at infinity.

\subsection{Gevrey asymptotics and Borel--Laplace summability}\label{sec:gevrey}

A formal series $\widehat A(\zeta)=\sum_{n\geq0}a_n\zeta^n$ is Gevrey of order $s\geq0$ when
\begin{dmath}\label{eq:s-Gevrey_condition}
 |a_n|\leq {CB^n\Gamma(1+sn),\qquad n\in\mathbb N}
\end{dmath},
for some $C,B>0$. For fixed $s$, this is equivalent to the usual bound with $(n!)^s$, after changing the constants \cite{aD15,mLR16}. The classes are nested, and order zero is equivalent to a positive convergence radius. Membership in the order-one class together with divergence does not establish that one is the smallest possible order.

An analytic function $A$ has $\widehat A$ as an ordinary asymptotic expansion at zero if the remainder after $N$ terms is $O(|\zeta|^N)$ for every fixed $N$. A \emph{strong Gevrey-$s$ expansion} requires the uniform estimate
\begin{dmath}\label{eq:strong-gevrey}
 \left|A(\zeta)-\sum_{n=0}^{N-1}a_n\zeta^n\right|
 \leq {CB^N\Gamma(1+sN)|\zeta|^N,\qquad N\in\mathbb N}
\end{dmath},
on each proper closed subsector, for sufficiently small $|\zeta|$. The constants may depend on that subsector but not on $N$ or $\zeta$. Sectors are understood on the logarithmic covering when their angular width requires it. For $s=0$, the estimate implies equality with the convergent series near zero. The corresponding convention for an inverse-power series is obtained by writing $f(t)=t^{-1}A(1/t)$; the remainder then contains $|t|^{-N-1}$.

For $s>0$, put $k=1/s$. A compatible Borel--Laplace normalisation in the small variable $\zeta$ is
\begin{dmath}\label{eq:k-Borel-normalisation}
 B_k(\xi)\defeq\sum_{n=0}^{\infty}\frac{a_n\xi^n}{\Gamma(1+n/k)}
\end{dmath},
followed by
\begin{dmath}\label{eq:k-Laplace-normalisation}
 (\mathcal S_{k,d}\widehat A\,)(\zeta)
 \defeq k\zeta^{-k}\int_0^{\infty\ee^{id}}
 \ee^{-(\xi/\zeta)^k}B_k(\xi)\,\xi^{k-1}\dd\xi
\end{dmath}.
Here $B_k$ must continue to a sector about the ray $d$, with growth at most $C\exp(a|\xi|^k)$ on proper subsectors, and all powers use compatible lifted logarithms. On the component $|d-\arg\zeta|<\pi/(2k)$, the integral converges when $|\zeta|^{-k}\cos(k(d-\arg\zeta))>a$. The Gamma integral verifies its coefficient normalisation; standard summability results give the strong Gevrey expansion \cite{Balser1994,mLR16}. For $k=1$, this convention is obtained from
\eqref{eq:Borel_transform} by applying $\mathcal B$ to
$t^{-1}\widehat A(t^{-1})$, rather than to
$\widehat A(t^{-1})$. The constant coefficient $a_0$ of
$\widehat A$ becomes the coefficient of $t^{-1}$ in this
constant-free inverse-power series, and hence
$B_1(0)=a_0$. The corresponding reconstruction satisfies
$t^{-1}(\mathcal S_{1,d}\widehat A)(t^{-1})
=(\mathcal L^d B_1)(t)$. 

Watson uniqueness identifies two functions with the same strong Gevrey-$s$ expansion on a common sector of opening greater than $\pi s$ \cite{Balser1994,mLR16}. On smaller sectors, nonzero terms such as $\exp(-a\zeta^{-1/s})$ can be smaller than every power in suitable directions. Consequently, agreement of coefficients alone is not sufficient to identify an AUF pairing with a Borel--Laplace sum: a common uniqueness domain or a direct integral identity is needed.

\subsection{Mellin transforms, rotated rays, and Mellin--Barnes integrals}\label{sec:Mellin_MB_preliminaries}

The Mellin transform is adapted to the multiplicative variable $z$. On a specified ray, with $\log z=\ln r+i\vartheta$, we use
\begin{dmath}\label{eq:directional-Mellin-transform}
 \mathcal M_\vartheta[f](s)
 \defeq {\int_0^{\infty\ee^{i\vartheta}}f(z)\,z^{s-1}\dd z
 =\ee^{i\vartheta s}\int_0^\infty f(r\ee^{i\vartheta})\,r^{s-1}\dd r}
\end{dmath}.
The jump functions below use the argument $s=-t$. If $f$ is locally integrable on the ray, $f(z)=O(|z|^a)$ at zero and $f(z)=O(|z|^b)$ at infinity on the ray, then $\mathcal M_\vartheta[f](-t)$ converges absolutely for $b<\RE t<a$, provided this interval is nonempty. The two bounds control different endpoints. A monomial has no ordinary full-ray Mellin strip, and avoiding finite poles does not remove these endpoint requirements \cite{Titchmarsh1948,Zemanian1987}.

On the positive ray, suppose $f$ is continuous, $\int_0^\infty|f(x)|x^{-c-1}\dd x<\infty$, and $y\mapsto\mathcal M_0[f](-c-iy)$ is integrable. Mellin inversion \cite{Titchmarsh1948,DLMF} then gives
\begin{dmath}\label{eq:Mellin-inversion}
 f(z)={\frac1{2\pi i}\int_{c-i\infty}^{c+i\infty}
 z^t\mathcal M_0[f](-t)\dd t,\qquad z>0}
\end{dmath}.
The spectral line is oriented upward. Integrals of this form, and their meromorphic continuations with suitably deformed contours, are Mellin--Barnes representations. Displacing the contour gives an exact finite residue sum and another contour integral when the joining segments vanish. Rightward displacement contributes minus the crossed residues; leftward displacement contributes plus them. An infinite residue expansion additionally requires control of the remaining integral \cite{ParisKaminski2001}.

We also use the reciprocal-Gamma representation
\begin{dmath}\label{eq:hankel-gamma}
 \frac1{\Gamma(s)}={\frac1{2\pi i}\int_{\mathcal H_R}\ee^w w^{-s}\dd w,
 \qquad s\in\mathbb C}
\end{dmath}.
The contour starts at negative infinity below the cut, follows a counterclockwise circle of fixed radius $R>0$, and returns above the cut; $-\pi<\arg w<\pi$ off the cut. With a nonzero loop radius, this formula is valid for every $s$, without an endpoint at the origin \cite[Eq.~5.9.2]{DLMF}. It will provide a complementary realisation when the vertical contour is unsuitable.

\section{Umbral Borel functionals and Laplace series}\label{sec:umbral_borel_functionals}

The Borel representatives connect the additive Laplace variable with the multiplicative Mellin variable. We write a generic representative as $\Delta(\ee^u;\zeta)$ and the same function in multiplicative coordinates as $\Delta(z;\zeta)$. The parameter $\zeta$ is held fixed when a transform in $u$ or $z$ is taken. The same symbol is used for functions related by this change of coordinates.
We use the term \emph{jump function} for the Mellin spectral representative $\jj(t;\zeta)$. The terminology comes from its expression, in the elementary examples, as the difference of opposite one-sided Laplace transforms on a common domain. It does not, by itself, denote a lateral Stokes discontinuity across a singular direction. Such discontinuities are defined separately in \eqref{eq:Stokes-discontinuity}.

\begin{definition}[Umbral Borel representatives]\label[definition]{def:umbral_Borel_functionals}
An umbral Borel representative is an ordinary function $\Delta(z;\zeta)$ for which $u\mapsto\Delta(\ee^u;\zeta)$ is holomorphic near $u=0$, together with the analytic continuation needed for the specified transform. The elementary family is
\begin{dmath}\label{eq:umbral-Borel-representative}
 \Delta(z;\zeta)=F(\zeta z),\qquad \zeta\neq0
\end{dmath},
where $F$ is holomorphic at $\zeta$. The initial classes studied here are polynomials and finite sums of rational functions and entire functions of finite exponential type. Their spectral realisations require the additional transform conditions stated below.
\end{definition}

The word \emph{functional} refers to the action represented by $\Delta$ on a ground state. The function itself is not a map from a test-function space to the scalars. Its spectral representative will define such a map through the pairing of \cref{sec:umbral_pairing}, where continuity on compatible spaces is established. We use \emph{umbral Borel functional} in this representational sense.

The formal Laplace series associated to an umbral Borel representative is
\begin{dmath}\label{eq:umbral-Laplace-series}
 \widehat\Delta(t;\zeta)\defeq
 \sum_{n=0}^{\infty}
 \left.\partial_u^n\Delta(\ee^u;\zeta)\right|_{u=0}t^{-n-1}
\end{dmath}.
Its Borel transform is the Taylor series of $\Delta(\ee^u;\zeta)$ at zero, so it is Gevrey of order at most one. This is an expansion in inverse powers of $t$, with $\zeta$ fixed. It must be distinguished from the expansion in $\zeta$ obtained after evaluating an umbral expression on a ground state. The latter has different coefficients and may have a different Gevrey order.

\subsection{Structure and classification of umbral Laplace series}\label{sec:classification}

Composition with $\ee^u$ changes both the singularities and the growth of a function profile. A polynomial becomes a finite sum of exponentials. A nonpolynomial entire profile remains entire but need not have exponential growth in $u$. A nonzero pole produces a logarithmic lattice of poles. These three mechanisms determine the following classification.

\begin{proposition}[Classification of umbral Laplace series]\label[proposition]{thr:classification_Laplace_umbral}
Fix $\zeta\neq0$ and suppose $F$ is holomorphic at both $0$ and $\zeta$. If $F$ is polynomial, the series \eqref{eq:umbral-Laplace-series} converges for sufficiently large $|t|$. If $F$ is entire and nonpolynomial, the series is divergent and belongs to the Gevrey-one class. If $F$ is meromorphic and has a nonzero pole, the series has exact Gevrey order one.
\end{proposition}

\begin{proof}
For $F(w)=\sum_{m=0}^M a_mw^m$, the coefficient of $t^{-n-1}$ is $\sum_m a_m\zeta^m m^n$. Summing the geometric series gives the rational value $\sum_m a_m\zeta^m/(t-m)$ for $|t|>M$.

Conversely, convergence of the Laplace series implies that $F(\zeta\ee^u)$ is entire of finite exponential type, by \cref{sec:formal_Borel_transform}. For large $|w|$, choose $u=\log(w/\zeta)$ with bounded imaginary part. An exponential bound in $|u|$ gives a polynomial bound in $|w|$. An entire function with this growth is polynomial, by Cauchy's inequalities. This excludes convergence for a nonpolynomial entire $F$.

A nonzero pole of $F$ yields a finite singularity of $F(\zeta\ee^u)$. Its Taylor series at zero therefore has a finite positive radius. If the coefficients of the corresponding Laplace series satisfied a Gevrey bound of order $s<1$, division by $n!$ would give an entire Borel series, a contradiction. The order-one upper bound follows from local analyticity.
\end{proof}

In the entire nonpolynomial case, Gevrey-one membership is an upper bound, not a claim of exact order. Divergence excludes order zero. Finite exponential type of $F(w)$ does not imply finite exponential type of $F(\zeta\ee^u)$, which is the condition relevant to convergence of the associated Laplace series.

\subsection{Parametric extensions}\label{sec:parametric}

Auxiliary powers and multiplicative prefactors lead to the more general representative
\begin{dmath}\label{eq:parametric-kernel}
 \Delta(z;\zeta)=\zeta^\beta z^\nu F(\zeta^\alpha z^\mu)
\end{dmath},
where $F$ is holomorphic at $\zeta^\alpha$ and logarithms are fixed for all powers. The parameters $\alpha,\beta$ act on the external variable; $\mu,\nu$ act on the auxiliary variable. The additive expression is simply $\Delta(\ee^u;\zeta)=\zeta^\beta\ee^{\nu u}F(\zeta^\alpha\ee^{\mu u})$.

Begin with $\mu>0$. Suppose
\begin{dmath*}
 \jj_F(s;\eta)\defeq\int_0^\infty x^{-s-1}F(\eta x)\dd x
\end{dmath*}
converges in a strip. The substitution $x=z^\mu$ in the Mellin transform of the \emph{complete} representative gives
\begin{dmath}\label{eq:affine-mellin}
 \jj(t;\zeta)=\frac{\zeta^\beta}{\mu}\,
 \jj_F\left(\frac{t-\nu}{\mu};\zeta^\alpha\right)
\end{dmath}.
Thus the auxiliary prefactor shifts the spectral variable, the auxiliary power rescales it, and the external prefactor remains outside the transform. The strip is the inverse image of the original strip under $s=(t-\nu)/\mu$.

For complex $\mu\neq0$, the same calculation is available on the lifted curve $u=v/\mu$, or $z=\exp(v/\mu)$, with $v$ increasing along the real line. On this curve $z^\mu=\ee^v$ and $\dd u=\dd v/\mu$. Formula \eqref{eq:affine-mellin} then refers to this specified curve, not to an unchanged positive ray. Correspondingly, a line $\RE s=c$ becomes the spectral contour $t=\nu+\mu(c+i\mathbb R)$ with its induced orientation. Any further deformation requires its own convergence and singularity checks.

\subsection{Laplace diagnostics and emergence of Mellin kernels}\label{sec:Laplace_diagnostics}

The following examples explain how a spectral representation emerges from the additive description. Opposite one-sided Laplace integrals, when both exist, can be combined into a bilateral integral. Under $z=\ee^u$, the latter is a Mellin integral. Polynomial representatives require a separate finite construction because their full-ray Mellin integrals have no convergence strip. "Polar profiles" in the following refers to functions with finite poles.

\subsubsection*{Polynomial profiles and finite spectral evaluation}

Let $\Delta(z;\zeta)=\sum_{m=0}^M a_m\zeta^m z^m$. Direct integration gives
\begin{dmath}\label{eq:polynomial-Laplace-transform}
 \mathcal L^0[\Delta(\ee^{\cdot};\zeta)](t)
 ={\sum_{m=0}^M\frac{a_m\zeta^m}{t-m},\qquad \RE t>M}
\end{dmath},
followed by meromorphic continuation. The finite indicial action has the Cauchy realisation
\begin{dmath}\label{eq:polynomial-pairing}
 \sum_{m=0}^M a_m\zeta^m\varphi(m)
 =\frac1{2\pi i}\sum_{m=0}^M
 \oint_{|t-m|=\epsilon_m}
 \frac{a_m\zeta^m\varphi(t)}{t-m}\dd t
\end{dmath},
where the circles are positively oriented and disjoint, and $\varphi$ is holomorphic in their interiors. This is an action by finitely many spectral evaluation functionals, as in the analytic-functional viewpoint of \cite{sternin1995borellaplace}. It reproduces polynomial umbral evaluation without assigning an ordinary Mellin transform to $z^m$.

\subsubsection*{Entire profiles: the exponential example}

Take $E(z;\zeta)=\ee^{\zeta z}$ and initially assume $\RE\zeta<0$. Fix $(-\zeta)^t=\exp(t\Log(-\zeta))$. For $\RE t<0$, Euler's Gamma integral \cite[Eq.~5.9.1]{DLMF} gives
\begin{dmath}\label{eq:jump-exponential}
 \jj_E(t;\zeta)
 ={
 \int_{-\infty}^{\infty}\ee^{-tu}E(\ee^u;\zeta)\dd u
 ={\int_0^\infty z^{-t-1}\ee^{\zeta z}\dd z
 =(-\zeta)^t\Gamma(-t)}
 }
\end{dmath}.
The condition on $t$ controls the endpoint $z=0$; exponential decay controls infinity.

The two parts of the real $u$-axis can also be evaluated separately:
\begin{dmath}\label{eq:exponential-one-sided-transforms}
 {\mathcal L^0[E(\ee^{\cdot};\zeta)](t)
 =(-\zeta)^t\Gamma(-t,-\zeta),\qquad
 \mathcal L^\pi[E(\ee^{\cdot};\zeta)](t)
 =-(-\zeta)^t\gamma(-t,-\zeta)}
\end{dmath},
where $\Gamma(a,x)$ and $\gamma(a,x)$ are the upper and lower incomplete Gamma functions. The first integral is entire in $t$; the second initially requires $\RE t<0$. Their difference is $\jj_E$, because the negative-ray integral is oriented from $0$ to $-\infty$. This comparison of opposite rays explains the name \emph{jump function}. It is not a lateral discontinuity across a finite Borel singularity: $E(\ee^u;\zeta)$ is entire in $u$.

The decay on the positive real $u$-axis does not extend to every ray in the right half-plane. Along $u=r\ee^{i\vartheta}$ with $\cos\vartheta>0$ and $\sin\vartheta\neq0$, the modulus of $\ee^u$ grows while its argument winds. There are then subsequences on which $\exp(\zeta\ee^u)$ grows faster than any exponential in $r$. In contrast, on rays with $\cos\vartheta<0$ it tends to one. The growth conditions of \cref{sec:Laplace_transform} must therefore be tested in the additive variable, not inferred from decay on a multiplicative ray.

For the parametric exponential in \eqref{eq:parametric-kernel}, the affine rule gives
\begin{dmath}\label{eq:parametric-exponential-jump}
 \jj(t;\zeta)=\frac{\zeta^\beta}{\mu}
 (-\zeta^\alpha)^{(t-\nu)/\mu}
 \Gamma\left(-\frac{t-\nu}{\mu}\right)
\end{dmath},
initially for $\RE(\zeta^\alpha)<0$ and $\RE((t-\nu)/\mu)<0$, on the auxiliary curve specified in \cref{sec:parametric}. 

\subsubsection*{Polar profiles: the rational example}

Let $R(z;\zeta)=(1-\zeta z)^{-1}$. For $\zeta\notin[0,\infty)$, there is no pole on the positive auxiliary ray. For $-1<\RE t<0$, the beta integral \cite[Sections~5.5 and~5.12]{DLMF}, first for $\zeta<0$ and then by continuation on this slit plane, gives
\begin{dmath}\label{eq:jump-rational}
 \jj_R(t;\zeta)
 ={
 \int_0^\infty\frac{z^{-t-1}}{1-\zeta z}\dd z
 ={(-\zeta)^t\Gamma(-t)\Gamma(1+t)
 =-\pi(-\zeta)^t\csc(\pi t)}
 }
\end{dmath}.
Here $\Log(-\zeta)$ again fixes the power. The upper and lower bounds on $\RE t$ come from zero and infinity, respectively. In additive coordinates the poles are
\begin{dmath}\label{eq:rational-pole-lattice}
 {u_k=\log(1/\zeta)+2\pi i k,\qquad
 \Res_{u=u_k}R(\ee^u;\zeta)=-1,\qquad k\in\mathbb Z}
\end{dmath},
where any one logarithm supplies a starting index for the lattice.

The same jump can be recovered from one period of this lattice. Put $\vartheta=\Arg(-\zeta)\in(-\pi,\pi)$. A positively oriented rectangle with heights $0$ and $2\pi i$ and vertical sides at $\RE u=\pm X$ encloses, for large $X$, exactly the pole
\begin{dmath*}
 u_*=-\ln|\zeta|+i(\pi-\vartheta)
\end{dmath*}.
For fixed $t$ in the strip, the left and right edge integrals are respectively $O(\ee^{X\RE t})$ and $O(\ee^{-X(1+\RE t)})$ and tend to zero. The top edge has opposite orientation and a factor $\ee^{-2\pi it}$ relative to the bottom. Hence
\begin{dmath}\label{eq:rat-period-rectangle}
 (1-\ee^{-2\pi it})\jj_R(t;\zeta)=-2\pi i\,\ee^{-tu_*}
\end{dmath},
which gives \eqref{eq:jump-rational}, including its sign. Expanding $(1-\ee^{-2\pi it})^{-1}$ in powers of $\ee^{-2\pi it}$ requires $\IM t<0$; the rectangle identity itself holds throughout the stated strip. It therefore avoids an unrestricted summation of infinitely many Stokes residues.

\subsubsection*{Why the spectral poles reproduce coefficients}

The same residue mechanism applies more generally. Suppose that, near the origin on the positive auxiliary ray,
\begin{dmath*}
 \Delta(z;\zeta)=\sum_{n=0}^{N-1}b_n(\zeta)z^n+O(z^N)
\end{dmath*},
and that its Mellin integral initially exists. Subtracting these finitely many terms on $0<z<1$ gives
\begin{dmath}\label{eq:endpoint-poles}
 \jj(t;\zeta)=\sum_{n=0}^{N-1}\frac{b_n(\zeta)}{n-t}
 +\int_0^1 z^{-t-1}\left[\Delta(z;\zeta)-\sum_{n=0}^{N-1}b_n(\zeta)\,z^n\right]\dd z
 +\int_1^\infty z^{-t-1}\Delta(z;\zeta)\dd z
\end{dmath}.
The subtracted integral extends holomorphically to $\RE t<N$. Wherever the last integral is holomorphic, continuation therefore gives
\begin{dmath}\label{eq:primary-spectral-residues}
 \Res_{t=n}\jj(t;\zeta)={-b_n(\zeta),\qquad 0\leq n<N}
\end{dmath}.
For $E$ and $R$, these residues are $-\zeta^n/n!$ and $-\zeta^n$. The minus sign is compensated by the sign of rightward Barnes displacement. The negative-integer poles of $\jj_R$ instead arise from the large-$z$ expansion of the rational representative.

This finite subtraction relates local coefficients to spectral residues without integrating the full Taylor series over the ray. Recovering the complete Mellin transform from an interpolation of the coefficients, as in Ramanujan's Master Theorem, requires additional analyticity and growth conditions \cite{Hardy1940Ramanujan}.

\subsection{Resurgent structure of umbral Borel functionals}\label{sec:resurgent_structure}

The first connection with resurgence concerns the auxiliary inverse-power series
$\widehat\Delta(t;\zeta)$, with $\zeta$ fixed. Its Borel transform is the germ of
$\Delta(\ee^u;\zeta)$ at $u=0$, by \eqref{eq:umbral-Laplace-series}.
In the classical continuation-based definition, a Gevrey-one series is resurgent
when its Borel germ is endlessly continuable \cite{Ecalle1981,dS16}.
For the profiles considered here, a sufficient, more concrete property is
continuation along every path avoiding a fixed locally finite set of points.
The following statement supplies this property for the initial AUF class.

\begin{proposition}[Resurgence of the auxiliary Laplace series]\label[proposition]{prop:umbral_resurgence_general}
Let $F=R+\sum_{\ell=1}^N E_\ell$, where $R$ is rational and each $E_\ell$ is entire
of finite exponential type. Fix $\zeta\neq0$ and its logarithm, $\mu\neq0$ and the
remaining parameters of \eqref{eq:parametric-kernel}, and suppose $F$ is
holomorphic at $\zeta^\alpha$. Then the formal inverse-power series with Borel
germ $\zeta^\beta\ee^{\nu u}F(\zeta^\alpha\ee^{\mu u})$ is Gevrey-one and resurgent.
If $a_1,\ldots,a_m$ are the nonzero poles of $R$, its possible finite Borel
singularities lie in
\begin{dmath}\label{eq:auxiliary-resurgent-set}
 {\Omega_\zeta=
 \bigcup_{j=1}^m
 \left\{\frac{\log(a_j/\zeta^\alpha)+2\pi i k}{\mu}:
 k\in\mathbb Z\right\}}
\end{dmath}.
The set is empty when there are no such poles. Changing the initial logarithm
in this expression only reindexes a lattice.
\end{proposition}
\begin{proof}
Local analyticity gives a bound $M\rho^{-n}$ for the Taylor coefficients of
$\Delta(\ee^u;\zeta)$ on a sufficiently small circle of radius $\rho$.
The inverse Borel coefficients are therefore bounded by $M n!\rho^{-n}$.
The entire terms and the factor $\ee^{\nu u}$ introduce no finite singularities.
The rational term is meromorphic in $u$, with poles only in
\eqref{eq:auxiliary-resurgent-set}; a profile pole at zero has no finite
preimage under the exponential. This finite union of arithmetic lattices has
no finite accumulation point, and the germ is regular at the origin by
hypothesis. It therefore continues along every path avoiding this set.
In particular, for paths of bounded length only finitely many obstructions
can be encountered, which gives endless continuability.
\end{proof}

Here Gevrey-one means membership in the order-one class, not necessarily
minimal order one. The divergence and exact-order statements of
\cref{sec:classification} remain separate. Entire Borel germs are included
in classical resurgence, but have no finite Borel singularities
\cite{dS16}. This does not force their inverse Laplace series to converge:
convergence requires finite exponential type in $u$, not merely in the profile
variable. Nor does endless continuability supply the directional growth bounds
needed for Laplace summation. The exponential example in
\cref{sec:Laplace_diagnostics} already exhibits this distinction.

For a simple parametric rational representative,
$\Delta(z;\zeta)=\zeta^\beta z^\nu/(1-a\zeta^\alpha z^\mu)$ with $a\neq0$, the locations and
residues are explicit:
\begin{dmath}\label{eq:parametric-rational-poles}
 {u_k=\frac{\log(1/(a\zeta^\alpha))+2\pi i k}{\mu},\qquad
 \Res_{u=u_k}\Delta(\ee^u;\zeta)
 =-\frac{\zeta^\beta}{\mu}\ee^{\nu u_k}}
\end{dmath}.
When a lateral contour deformation is justified, these local singularities
give the exponential terms in \eqref{eq:stokes_simple_poles}.

\begin{example}[One Borel pole, coefficient growth and a Stokes term]\label[example]{ex:auxiliary-resurgent-pole}
Take $R(z;\zeta)=(1-\zeta z)^{-1}$ and fix $\zeta=\ee^{-a}$, with $a>0$.
The additive representative $(1-\ee^{u-a})^{-1}$ has poles $a+2\pi i k$ and
residue $-1$ at each pole. Write its inverse Borel coefficients as
$c_n(a)=\left.\partial_u^n R(\ee^u;\ee^{-a})\right|_{u=0}$.
Subtracting $(a-u)^{-1}$ removes the unique nearest pole. Cauchy's estimate
on any circle of radius between $a$ and $\sqrt{a^2+4\pi^2}$ gives
\begin{dmath}\label{eq:auxiliary-rational-large-order}
 {c_n(a)\sim\frac{n!}{a^{n+1}},\qquad n\to\infty}
\end{dmath}.
Let $S_\pm(t)$ be the lateral Laplace transforms on the two sides of the
positive $u$-axis, with signs denoting increasing and decreasing angles.
A narrow sector contains only the pole $u=a$; on its distant arcs the rational
representative decays exponentially. For $t>0$ the deformation is thus
legitimate, and \eqref{eq:stokes_simple_poles} gives
\begin{dmath}\label{eq:auxiliary-rational-stokes}
 {S_+(t)-S_-(t)=2\pi i\,\ee^{-at},\qquad t>0}
\end{dmath}.
The same Borel pole determines the factorial coefficient growth and the
exponentially small lateral difference. This is a nontrivial resurgent
connection before any ground state is introduced.
\end{example}

The period identity \eqref{eq:rat-period-rectangle} shows how
the rational singular lattice enters the Mellin representative
$\jj$. This relation does not identify $\jj$ with the lateral
discontinuity calculated above or with a Stokes automorphism.

After pairing, a second resurgence question concerns
$\widehat{\mathcal A}(\zeta)$ and its Borel transform $B_k(\xi)$.
These include the ground-state weights and are not the auxiliary series and
germ just considered. In the quartic model below, an entire additive representative
gives an evaluated series with a finite Borel singularity in $\xi$ and an
exact Stokes connection. Conversely, the lognormal \cref{ex:lognormal-state} has an evaluated
expansion beyond every finite Gevrey order. The resurgent auxiliary
construction and these broader pairing domains are therefore compatible:
resurgence of the evaluated series must be established separately.

\subsection{From sectorial Laplace transforms to the jump function}\label{sec:jump_kernel}

The spectral object entering the pairing is defined as follows.

\begin{definition}[Spectral jump function]\label[definition]{def:jump_bilateral}
Choose an auxiliary ray $\ell_\vartheta=\ee^{i\vartheta}\mathbb R_{>0}$ and its logarithm. Whenever the integral converges absolutely in a nonempty open strip, define
\begin{dmath}\label{eq:general-jump}
 \jj_\vartheta(t;\zeta)\defeq {
 \int_{\ell_\vartheta}z^{-t-1}\Delta(z;\zeta)\dd z
 =\int_{\mathbb R+i\vartheta}\ee^{-tu}\Delta(\ee^u;\zeta)\dd u
 }
\end{dmath},
with both paths oriented in the direction of increasing $\RE u$. The jump function is this transform and its analytic continuation, with the ray, logarithms and continuation path understood. The ray subscript is suppressed when these data have been fixed.
\end{definition}

When $\Delta$ is holomorphic in $\zeta$, an integrable bound uniform on compact subsets of the initial parameter and spectral domains gives joint holomorphy there. If the integrand is holomorphic between two rays and the endpoint arcs vanish, the transforms agree. Crossing a pole or cut requires the corresponding contour contribution. Finite sums may be transformed term by term on a common convergence domain; a decomposition whose summands have incompatible strips requires a separate argument.

For the elementary jump functions, the external logarithmic dependence can be made completely explicit. Start at $\zeta<0$ with lifted argument $\theta=\pi$, and continue the logarithm as
\begin{dmath*}
 \log(-\zeta)={\ln r+i(\theta-\pi),\qquad r>0,\quad\theta\in\mathbb R}
\end{dmath*}.
The lifted continuations are consequently
\begin{dmath}\label{eq:elementary-lifted-jumps}
 {\widetilde \jj_E(t;r,\theta)=\ee^{t[\ln r+i(\theta-\pi)]}\Gamma(-t),\qquad
 \widetilde \jj_R(t;r,\theta)=-\pi\ee^{t[\ln r+i(\theta-\pi)]}\csc(\pi t)}
\end{dmath}.
The original positive-ray formulas occupy $\pi/2<\theta<3\pi/2$ for $E$, and $0<\theta<2\pi$ for $R$, with their respective $t$-strips. Outside these initial domains, \eqref{eq:elementary-lifted-jumps} specifies continuation, not convergence of the unchanged auxiliary integrals.

The subsequent Barnes integral depends on the product of $\,\widetilde \jj$ with the ground state, so its parameter domain requires a separate convergence test. First, \cref{sec:spectral_duality} develops the Gamma-factor relation between the two elementary jumps.

\section{Spectral duality}\label{sec:spectral_duality}

The exponential and rational representatives of \cref{sec:Laplace_diagnostics} have jump functions related by a factor independent of the external parameter. We establish this relation and its extension to suitable superpositions here; its effect on ground-state evaluation is treated in \cref{sec:umbral_pairing}.

\begin{theorem}[Elementary spectral duality]\label[theorem]{thm:reflection_umbral_pairing}
For $E(z;\zeta)=\ee^{\zeta z}$ and $R(z;\zeta)=(1-\zeta z)^{-1}$, the jump functions satisfy
\begin{dmath}\label{eq:elementary-duality}
 \jj_R(t;\zeta)=\Gamma(1+t)\,\jj_E(t;\zeta)
\end{dmath},
initially for $\RE\zeta<0$ and $-1<\RE t<0$. The identity persists under common analytic continuation with compatible logarithmic determinations.
\end{theorem}
\begin{proof}
By \eqref{eq:jump-exponential} and \eqref{eq:jump-rational}, the two sides are $(-\zeta)^t\Gamma(-t)\Gamma(1+t)$. Euler's reflection formula gives the equivalent rational expression $-\pi(-\zeta)^t\csc(\pi t)$ \cite[Eq.~5.5.3]{DLMF}. The identity theorem gives the continuation statement, interpreted meromorphically at spectral poles.
\end{proof}

We call the relation \eqref{eq:elementary-duality} \emph{spectral duality}. Its origin can also be seen before taking the Mellin transform:
\begin{dmath}\label{eq:elementary-laplace-duality}
 \int_0^\infty\ee^{-v}\ee^{vw}\dd v
 ={\frac1{1-w},\qquad \RE w<1}
\end{dmath}.
Integration multiplies each Taylor coefficient by $n!$, converting the exponential series into the geometric series. The following proposition establishes the corresponding Mellin factor $\Gamma(1+t)$ for general profiles. Here $\jj_E(t;\zeta)$ denotes the jump of $E(\zeta z)$, and similarly for its companion.

\begin{proposition}[Gamma identity on a common Mellin strip]\label[proposition]{prop:gamma-duality}
Fix $\zeta\neq0$. Suppose the profile $E$ is continuous on $\zeta\mathbb R_{>0}$ and define
\begin{dmath}\label{eq:damped-companion}
 R_E(w)\defeq\int_0^\infty\ee^{-v}E(vw)\dd v
\end{dmath},
assuming absolute convergence for every $w=\zeta x$, $x>0$. If
\begin{dmath}\label{eq:duality-absolute}
 \int_0^\infty x^{-c-1}|E(\zeta x)|\,\dd x<\infty
 \qquad\text{for every }c\in I,
 \qquad I\subset(-1,0)
\end{dmath},
where $I$ is a nonempty open interval, then $R_E(\zeta x)$ is absolutely Mellin-transformable in the same strip and
\begin{dmath}\label{eq:general-gamma-duality}
 {\jj_{R_E}(t;\zeta)=\Gamma(1+t)\,\jj_E(t;\zeta),\qquad \RE t\in I}
\end{dmath}.
\end{proposition}
\begin{proof}
For $c\in I$, substitution $y=vx$ gives
\begin{dmath}\label{eq:gamma-duality-fubini}
 \int_0^\infty\!\int_0^\infty
 x^{-c-1}\ee^{-v}|E(v\zeta x)|\,\dd v\dd x
 =\Gamma(1+c)\int_0^\infty y^{-c-1}|E(\zeta y)|\,\dd y
\end{dmath}.
This finite bound justifies Fubini's theorem. The same substitution without absolute values produces the factor $\int_0^\infty\ee^{-v}v^t\dd v=\Gamma(1+t)$, proving \eqref{eq:general-gamma-duality}.
\end{proof}

These formulas use the positive auxiliary ray. On a rotated ray, the substitution in \eqref{eq:directional-Mellin-transform} supplies the same phase factor on both sides. Local uniform versions of the bounds give holomorphic dependence on $\zeta$; any further identity follows by common continuation, not by assuming that the original integrals still converge.

The affine rule of \cref{sec:parametric} also preserves the relation. For representatives $E_{\mathrm{par}} = \zeta^\beta z^\nu E(\zeta^\alpha z^\mu)$ and $R_{E,{\mathrm{par}}}=\zeta^\beta z^\nu R_E(\zeta^\alpha z^\mu)$, their jumps obey
\begin{dmath}\label{eq:parametric-spectral-duality}
 \jj_{R_{E,\mathrm{par}}}(t;\zeta)
 ={\Gamma(1+s)\,\jj_{E_{\mathrm{par}}}(t;\zeta),
 \qquad s=\frac{t-\nu}{\mu}}
\end{dmath}
on the common transformed strip and auxiliary curve specified there. In particular, both jumps retain the factor $\zeta^\beta/\mu$.

\subsection{P{\'o}lya realisation of spectral duality}\label{sec:polya}

P{\'o}lya's exponential superposition produces a resolvent companion by replacing each exponential with its rational counterpart. This gives a coefficient correspondence near zero; a common Mellin identity requires an additional integrability estimate.

Let $E(w)=\sum_{n=0}^\infty a_nw^n$ be entire of finite exponential type. Its \emph{exterior Borel transform} is the analytic germ
\begin{dmath}\label{eq:polya-transform}
 \mathcal P_E(\lambda)\defeq\sum_{n=0}^\infty n!a_n\lambda^{-n-1}
\end{dmath},
convergent for sufficiently large $|\lambda|$. This is $\mathcal B^{-1}$ applied to the Taylor series of $E$ in the convention of \cref{sec:formal_Borel_transform}, not the additive representative $\Delta(\ee^u;\zeta)$. P{\'o}lya's formula gives \cite[Chapter~20]{RubelColliander95}
\begin{dmath}\label{eq:polya-representation}
 {E(w)=\frac1{2\pi i}\oint_\gamma
 \ee^{\lambda w}\mathcal P_E(\lambda)\dd\lambda,
 \qquad n!a_n=\frac1{2\pi i}\oint_\gamma
 \lambda^n\mathcal P_E(\lambda)\dd\lambda}
\end{dmath}.
Here $\gamma$ is initially a sufficiently large positively oriented circle. The formulas follow by integrating the convergent exterior series term by term. The contour may subsequently be deformed where $\mathcal P_E$ continues analytically, provided no singularity is crossed. The contour density $\mathcal P_E$ is an ordinary analytic function.

Define the \emph{P{\'o}lya-resolvent companion} by
\begin{dmath}\label{eq:polya-resolvent}
 {R_E(w)\defeq\frac1{2\pi i}\oint_\gamma
 \frac{\mathcal P_E(\lambda)}{1-\lambda w}\dd\lambda
 =\sum_{n=0}^\infty n!a_nw^n}
\end{dmath},
initially for $|w|\,\max_\gamma|\lambda|<1$. The exponential-type bound ensures a positive radius of convergence. Integrating \eqref{eq:polya-representation} against $\ee^{-v}\dd v$ shows that this germ agrees with \eqref{eq:damped-companion} for sufficiently small $|w|$. Its continuation may be rational, algebraic, logarithmic or entire. In particular, near zero,
\begin{dmath}\label{eq:resolvent-exterior-transform}
 R_E(w)=w^{-1}\mathcal P_E(1/w)
\end{dmath},
with the value at zero filled in by continuity. This identity relates its singularities to the continued exterior transform, rather than to an arbitrary choice of enclosing contour.

For a direct spectral construction, the exponential components should decay together on the selected auxiliary ray. The next sufficient condition makes this requirement explicit.

\begin{theorem}[P{\'o}lya--Mellin spectral duality]\label[theorem]{thm:polya_spectral_duality}
Suppose $\gamma$ is a compact rectifiable contour for \eqref{eq:polya-representation}, with $\mathcal P_E$ holomorphic near $\gamma$, and
\begin{dmath}\label{eq:polya-decay}
 {\RE(\lambda\zeta)\leq-\delta<0,\qquad \lambda\in\gamma}
\end{dmath},
for some $\delta>0$. Then the positive-ray Mellin transforms of $E(\zeta z)$ and $R_E(\zeta z)$ converge absolutely for $-1<\RE t<0$ and satisfy \eqref{eq:general-gamma-duality}.
\end{theorem}
\begin{proof}
The exponentials are bounded by $\ee^{-\delta x}$, and $|(1-\lambda\zeta x)^{-1}|\leq(1+\delta x)^{-1}$ for $x>0$. Since $\mathcal P_E$ is bounded on the finite contour, both P{\'o}lya integrals can be interchanged with the Mellin integral in the stated strip. The common factor is
\begin{dmath}\label{eq:polya-spectral-factor}
 H(t;\zeta)=\frac1{2\pi i}\oint_\gamma
 (-\lambda\zeta)^t\mathcal P_E(\lambda)\dd\lambda
\end{dmath},
where the principal logarithm is used. The exponential transform is $\Gamma(-t)H(t;\zeta)$; the transform of its companion has the additional factor $\Gamma(1+t)$.
\end{proof}

This sufficient decay condition is not necessary: the examples include profiles whose Mellin integrability emerges only after superposition.

\begin{example}[Discrete support and rational companions]\label[example]{ex:polya_discrete}
For $a,b\neq0$, the profile $E(w)=\ee^{aw}+\ee^{bw}$ has exterior transform $(\lambda-a)^{-1}+(\lambda-b)^{-1}$ and companion
\begin{dmath*}
 R_E(w)=\frac1{1-aw}+\frac1{1-bw}
\end{dmath*}.
On any nonempty common decay domain $\RE(a\zeta)<0$, $\RE(b\zeta)<0$, the spectral identity follows componentwise. The rational poles occur at the reciprocal support points $1/a$ and $1/b$.

Higher-order poles arise from exponential-polynomial profiles. For $E(w)=w^m\ee^{aw}$, $m\in\mathbb N$, $a\neq0$, direct integration gives
\begin{dmath}\label{eq:exp-polynomial-companion}
 {R_E(w)=\frac{m!w^m}{(1-aw)^{m+1}},\qquad
 \jj_E(t;\zeta)=\zeta^m(-a\zeta)^{t-m}\Gamma(m-t)}
\end{dmath}.
The common Mellin strip is $-1<\RE t<m$ when $\RE(a\zeta)<0$, and the companion transform is $\Gamma(1+t)\,\jj_E$. The factor $\zeta^m$ records the polynomial part of the original representative.
\end{example}

\begin{example}[Continuous support on an oscillatory ray]\label[example]{ex:polya_continuous}
A continuous superposition gives
\begin{dmath}\label{eq:continuous-dual}
 {E(w)=\frac12\int_{-1}^1\ee^{\lambda w}\dd\lambda
 =\frac{\sinh w}{w},\qquad
 R_E(w)=\frac12\int_{-1}^1\frac{\dd\lambda}{1-\lambda w}
 =\frac1{2w}\log\frac{1+w}{1-w}}
\end{dmath},
where the logarithm is continued from zero and both values at zero are one. The companion is logarithmic rather than rational. The symmetric support cannot satisfy a common strict decay condition. Nevertheless, at $w=ix$ the profiles are $\sin x/x$ and $\arctan x/x$. Both are bounded at zero and $O(x^{-1})$ at infinity, so their Mellin integrals converge absolutely for $-1<\RE t<0$.

Put $s=-t$, with $0<\RE s<1$. Integration by parts reduces the first transform to the cosine Mellin integral and the second to a beta integral, giving \cite[Sections~5.9 and~5.12]{DLMF}
\begin{dmath}\label{eq:continuous-imaginary-jumps}
 {\jj_E(-s;i)=\frac{\Gamma(s)\cos(\pi s/2)}{1-s},\qquad
 \jj_{R_E}(-s;i)=\frac{\pi}{2(1-s)\sin(\pi s/2)}}
\end{dmath}.
The endpoint terms vanish in this strip. The cosine integral is conditionally convergent; inserting $\ee^{-\epsilon x}$ and then letting $\epsilon\downarrow0$ gives its Gamma value. Reflection now yields $\jj_{R_E}(-s;i)=\Gamma(1-s)\,\jj_E(-s;i)$. Here the decay needed for absolute Mellin integration appears after the exponential superposition has been performed.
\end{example}

\begin{example}[Bessel profile and an algebraic companion]\label[example]{ex:bessel_kernel}
Take $E(w)=J_0(w)$. Its angular representation $J_0(w)=\pi^{-1}\int_0^\pi\ee^{iw\cos v}\dd v$ supplies another oscillatory superposition \cite[Eq.~10.9.2]{DLMF}. Averaging $(1-iw\cos v)^{-1}$ instead, or comparing its Taylor coefficients, gives
\begin{dmath}\label{eq:bessel-companion}
 {E(w)=J_0(w),\qquad R_E(w)=(1+w^2)^{-1/2}}
\end{dmath},
with the square root equal to one at the origin. For $\zeta>0$, the standard Bessel Mellin integral and a beta integral give \cite[Sections~5.12 and~10.22]{DLMF}
\begin{dmath}\label{eq:bessel-companion-jumps}
 {\jj_E(t;\zeta)=2^{-t-1}\zeta^t
 \frac{\Gamma(-t/2)}{\Gamma(1+t/2)},\qquad
 \jj_{R_E}(t;\zeta)=\frac{\zeta^t}{2\sqrt\pi}
 \Gamma(-t/2)\Gamma((1+t)/2)}
\end{dmath}.
The common absolute strip is $-1/2<\RE t<0$: the Bessel profile is bounded at zero and $O(x^{-1/2})$ on the positive ray at infinity, whereas its companion is $O(x^{-1})$ \cite[Section~10.17]{DLMF}. The Gamma duplication formula verifies spectral duality. As in the preceding example, this calculation does not require strict decay of each exponential component.
\end{example}

\begin{example}[Wright and Mittag--Leffler correspondence]\label[example]{ex:wright_function}
For $a,b>0$, consider the generalized Bessel function of Wright \cite{emW34} and its companion,
\begin{dmath}\label{eq:wright-mittag}
 {W_{a,b}(w)=\sum_{n=0}^\infty\frac{w^n}{n!\Gamma(an+b)},\qquad
 E_{a,b}(w)=\sum_{n=0}^\infty\frac{w^n}{\Gamma(an+b)}}
\end{dmath}.
Both functions are entire, and $W_{a,b}$ has order $1/(1+a)<1$, so it belongs to the exponential-type class. Factorial multiplication gives the Mittag--Leffler function. In this case \eqref{eq:damped-companion} holds for every $w$: the sum of the absolute term integrals is $\sum_{n\geq0}|w|^n/\Gamma(an+b)$, which converges uniformly on compact sets. Thus a resolvent companion can itself be entire. This coefficient and integral correspondence does not assert a common Mellin strip for every $a,b$. Its AUF realisation will be developed with the inverse-Gamma examples in \cref{sec:examples}.
\end{example}

Spectral duality therefore provides a precise way to relate different auxiliary representatives. It is first an identity between jump functions on a common analytic domain. The next section introduces the ground state and establishes how this identity becomes a transformation preserving the analytic pairing.

\section{The Analytic Umbral Pairing}\label{sec:umbral_pairing}

The jump function describes the spectral content of an umbral Borel representative. To evaluate that representative, we must specify how its spectral components act on a ground state. The pairing introduced here supplies this action as an analytic function, with a stated contour and domain. Its relation to coefficientwise umbral evaluation is then established by estimating the remainder in a finite contour displacement. The examples will show why existence of the pairing and recovery of an ordinary power series are distinct questions.

\subsection{Admissible ground state functions}\label{sec:admissible_ground_states}

In the formal rule, only values such as $\varphi(n)$ occur. An analytic ground state supplies a function of the complex spectral variable, not just these samples. We allow holomorphic or meromorphic states, and states with specified branches, provided their product with the jump function is well defined along the chosen contour. Holomorphy at the sampled indices is an additional requirement when the ordinary indicial rule is to be recovered.

Gamma products give an important family because their poles and vertical growth are explicit. For positive real slopes, a pure quotient satisfies
\begin{dmath}\label{eq:gamma-quotient-rate}
 \left|\frac{\prod_j\Gamma(a_j(c+iy)+b_j)}
 {\prod_k\Gamma(\alpha_k(c+iy)+\beta_k)}\right|
 \asymp (1+|y|)^{M(c)}
 \exp\left[-\frac\pi2\left(\sum_j a_j-\sum_k\alpha_k\right)|y|\right]
\end{dmath}
as $|y|\to\infty$, away from poles. The real power exponent $M(c)$ is determined by the shifts, and multiplicities are included by repeating factors. This is a consequence of vertical Stirling estimates \cite[Section~5.11]{DLMF}. Algebraic branch factors change the power bound; other slopes require separate estimates. In particular, $\Gamma(1+t)$ decays exponentially on vertical lines, whereas its reciprocal grows. Admissibility is decided by the \emph{complete product} with the jump function, not by the algebraic form of the ground state.

For the functional interpretation, fix a connected open neighbourhood $U$ of a spectral contour $\mathcal C$ and a positive weight $w$ on that contour. A concrete vector space of states is
\begin{dmath}\label{eq:weighted-state-space}
 {X_w=\{\varphi\in\mathcal O(U):p_w(\varphi)<\infty\},\qquad
 p_w(\varphi)=\sup_{t\in\mathcal C}\frac{|\varphi(t)|}{w(t)}}
\end{dmath},
equipped with $p_w$ and the compact-open seminorms of $\mathcal O(U)$. For fixed $U$, $\mathcal C$ and $w$, this is a Fr\'echet space \cite{BonetVukotic2017}. Indeed, a compact exhaustion of $U$ makes the topology countably generated. A Cauchy sequence converges locally uniformly to a holomorphic function, while its weighted restrictions converge uniformly on $\mathcal C$; the limits agree pointwise, proving completeness. Different contours or growth bounds are treated by their corresponding spaces, with contour invariance subject to the hypotheses of \cref{prop:pairing-basic}. Gamma quotients satisfying the bounds are examples, not the whole state space.

\subsection{Umbral pairing as a Mellin--Barnes integral}\label{sec:pairing-mb}

The spectral variable $t$ is integrated over a contour rather than restricted to integer values. Consequently, factors such as $(-\zeta)^t$ depend on a continuous choice of argument. We use the lift fixed in \cref{sec:jump_kernel}:
\begin{dmath}\label{eq:pairing-logarithm}
 {L=\log(-\zeta)=\ln r+i(\theta-\pi),\qquad
 \zeta=r\ee^{i\theta},\quad r>0,\quad\theta\in\mathbb R}
\end{dmath}.
The points $(r,\theta)$ and $(r,\theta+2\pi)$ project to the same complex number, but the spectral powers need not have the same value. The natural initial domain is therefore a domain on the logarithmic covering, usually an angular sector. Holomorphy there means holomorphy in the local coordinate $\ln r+i\theta$.

\begin{definition}[Admissible data and analytic pairing]\label[definition]{def:admissible-pairing-data}\label[definition]{def:analytic-umbral-pairing}
Choose a continued jump function $\widetilde \jj$, a ground state $\varphi$, fixed branches and an oriented Barnes contour $\mathcal C$. Let $\mathcal S$ be an open domain on the logarithmic covering. The data are admissible on $\mathcal S$ if the complete density is defined on $\mathcal C$, is holomorphic in the external parameter $\zeta$, and has an integrable majorant on $\mathcal C$ uniformly on every compact subset of $\mathcal S$. The analytic umbral pairing is then
\begin{dmath}\label{eq:analytic-pairing}
 \widetilde{\mathcal A}(r,\theta)
 \defeq {\langle\Delta,\varphi\rangle_{\mathcal C}(r,\theta)
 =\frac1{2\pi i}\int_{\mathcal C}
 \widetilde \jj(t;r,\theta)\,\varphi(t)\dd t}
\end{dmath}.
The contour avoids nonremovable singularities of the complete density and specifies which singularities lie on either side. Vertical lines are oriented upward.
\end{definition}

For the primary exponential and rational pairings, the nonnegative integer poles of the jump function lie to the right. The rational negative-integer poles lie to the left; poles of the state impose further restrictions. Cancellations are assessed in the complete density. A contour that becomes pinched as parameters vary cannot be retained without a separate limiting argument.

\paragraph{How the convergence sector arises.}
For a complete density of the form $\widetilde \jj(t;r,\theta)\,\varphi(t)=\ee^{tL}G(t)$, with $G$ independent of the external parameter, suppose
\begin{dmath}\label{eq:vertical-density-bound}
 {|G(c+iy)|\leq C(1+|y|)^m\ee^{-a|y|},\qquad a>0}
\end{dmath}.
Since $|\ee^{(c+iy)L}|=r^c\ee^{-y(\theta-\pi)}$, both tails are integrable when $|\theta-\pi|<a$.
The estimate is uniform on compact subsets of this sector. Positive and negative imaginary tails impose its two boundaries; asymmetric decay gives an asymmetric interval. For example, the rational representative paired with $\Gamma(1+t)$ has $a=3\pi/2$. Its vertical integral thus converges on $-\pi/2<\theta<5\pi/2$, a sector wider than one cut plane. Boundary rays require a separate test of the remaining power factors.

A planar determination is the restriction
\begin{dmath}\label{eq:pairing-planar-section}
 {\mathcal A_\phi(\zeta)
 =\widetilde{\mathcal A}\bigl(|\zeta|,\arg_\phi\zeta\bigr),
 \qquad \phi-2\pi<\arg_\phi\zeta<\phi}
\end{dmath},
where the corresponding lifted point belongs to the domain of the pairing or its established continuation. Changing $\phi$ changes this restriction, not the underlying construction. The domain of convergence of the displayed integral, the domain reached by analytic continuation, and the sector of a proved asymptotic expansion must be kept distinct. None is assumed to be the entire covering.

\begin{proposition}[Holomorphy and continuous action]\label[proposition]{prop:pairing-basic}
An admissible pairing is holomorphic on $\mathcal S$. For the space $X_w$ in \eqref{eq:weighted-state-space}, suppose
\begin{dmath}\label{eq:pairing-seminorm}
 M_K=\int_{\mathcal C}\sup_{(r,\theta)\in K}
 |\widetilde \jj(t;r,\theta)|\,w(t)\,|\dd t|<\infty
\end{dmath}
for each compact $K\subset\mathcal S$. Then the pairing is a continuous linear map from $X_w$ to $\mathcal O(\mathcal S)$, with
\begin{dmath}\label{eq:pairing-continuity-bound}
 \sup_K|\widetilde{\mathcal A}|\leq\frac{M_K}{2\pi}p_w(\varphi)
\end{dmath}.
Admissible contour deformations preserve its value if they cross no singularity and their connecting integrals tend to zero.
\end{proposition}
\begin{proof}
The locally uniform majorant permits contour integration in the external parameter to pass through the spectral integral. Morera's theorem gives holomorphy. The displayed bound follows from the integral triangle inequality. The deformation statement follows from Cauchy's theorem on truncated contours and the stated limit of the connectors.
\end{proof}

At fixed external parameter, the representative defines a continuous functional on $X_w$, justifying the term \emph{umbral Borel functional} \cite{GelfandShilov1968,Zemanian1987}. This does not identify a particular Gelfand--Shilov space or exhaust the continuous dual. The pairing is also linear in representatives on a common admissible domain; polynomials retain the finite Cauchy action of \eqref{eq:polynomial-pairing}.

\paragraph{Reduction to a weighted auxiliary integral.}
A state sometimes has the complex-moment representation
\begin{dmath}\label{eq:moment-state}
 {\varphi(t)=\int_0^\infty x^t\dd\omega(x),\qquad
 \int_0^\infty x^c\dd\,|\omega|(x)<\infty}
\end{dmath}
on $\mathcal C_c=c+i\mathbb R$. If $\jj(c+iy;\zeta)$ is integrable in $y$ and its Mellin inversion recovers $\Delta(x;\zeta)$ on the support of the measure, then
\begin{dmath}\label{eq:moment-evaluation}
 \langle\Delta,\varphi\rangle_{\mathcal C_c}(\zeta)
 =\int_0^\infty\Delta(x;\zeta)\dd\omega(x)
\end{dmath}.
Indeed, the absolute double integral is bounded by the product of the two stated finite integrals. Fubini's theorem and \eqref{eq:Mellin-inversion} prove the identity. Agreement only at integer moments would not prove this identification.

There is a complementary sufficient criterion. If $\Delta$ has an absolutely convergent positive-ray Mellin transform on $\RE t=c$ and $\varphi(c+iy)$ is integrable in $y$, direct Fubini gives
\begin{dmath}\label{eq:mellin-reduction}
 {\langle\Delta,\varphi\rangle_{\mathcal C_c}(\zeta)
 =\int_0^\infty\Delta(x;\zeta)W_\varphi(x)\frac{\dd x}{x},\qquad
 W_\varphi(x)=\frac1{2\pi i}\int_{\mathcal C_c}x^{-t}\varphi(t)\dd t}
\end{dmath}.
When both descriptions apply with $\dd\omega(x)=\rho(x)\dd x$, one has $W_\varphi(x)=x\rho(x)$. These are sufficient Mellin--Parseval hypotheses \cite{Titchmarsh1948,DLMF}, not consequences of convergence of the product $\jj\varphi$ alone.

Parameter differentiation is justified by a local integrable majorant for the differentiated kernel. An additional improper integral in the external variable requires its own interchange estimate. For an already identified auxiliary realization on a measure or contour, absolute integrability of the joint integrand is a sufficient condition. This permits integration of an exact representative, not termwise integration of its formal expansion. The Gaussian \cref{ex:gaussian-operational} in \cref{sec:examples} will illustrate the distinction.

\subsection{A Hankel realisation for inverse-Gamma ground states}\label{sec:Hankel_reduction}

The entire state $1/\Gamma(1+t)$ illustrates why a vertical integral is not the only useful realization of an AUF pairing. Its vertical growth may cancel the decay supplied by a jump function. Hankel's formula instead represents the same state on a loop contour whose tails are controlled by $\ee^w$. This leads to a complementary direct construction for an inverse-Gamma family.

\begin{theorem}[Fixed-radius Hankel realisation]\label[theorem]{thm:hankel}
Let $\varphi(t)=1/\Gamma(b+\rho t)$, with $b\in\mathbb C$, $\rho>0$, and let $F(w)=\sum_{n\geq0}a_nw^n$ have positive Taylor radius. For $\mu>0$ and $\nu\in\mathbb C$, the pairing with the representative $z^\nu F(\zeta z^\mu)$ has the Hankel realization
\begin{dmath}\label{eq:hankel-evaluation}
 \mathcal A_H(\zeta)
 =\frac1{2\pi i}\int_{\mathcal H_R}
 \ee^w w^{-b-\rho\nu}F(\zeta w^{-\rho\mu})\dd w
\end{dmath},
where $\mathcal H_R$ and its logarithm are as in \eqref{eq:hankel-gamma}. For each compact $\zeta$-set, choose $R$ sufficiently large that the substituted arguments remain strictly inside the Taylor disc. The value is independent of such $R$ and equals the entire function
\begin{dmath}\label{eq:hankel-series}
 \mathcal A_H(\zeta)=\sum_{n=0}^\infty
 \frac{a_n\zeta^n}{\Gamma(b+\rho\nu+\rho\mu n)}
\end{dmath}.
\end{theorem}
\begin{proof}
The substituted Taylor series is uniformly absolutely convergent on the contour, and $\ee^w$ controls its tails. Termwise integration is therefore justified. Hankel's formula gives each coefficient \cite[Eq.~5.9.2]{DLMF}. The exponential coefficient bound for $F$ and real-direction Stirling estimates give convergence uniformly on compact $\zeta$-sets. Between two admissible radii the integrand is holomorphic off the fixed cut, so Cauchy's theorem gives the same value. These local constructions thus define one entire function.
\end{proof}

The contour radius is not contracted through singularities of the substituted profile. The branch relation is $\log z=-\rho\log w$, so complex $\nu$ is retained without restriction. An external prefactor $\zeta^\beta$ and substitution $\zeta\mapsto\zeta^\alpha$ can then be restored on their chosen lift; the resulting function need not be entire in that external variable.

Agreement with a vertical pairing requires more than agreement of coefficients. It follows by substituting Hankel's representation of the state when Mellin inversion holds on the auxiliary points $z=w^{-\rho}$ and the resulting double integral is absolutely convergent. For the elementary representatives with $\mu=1$, $\nu=0$, this can be verified by deforming the Hankel tails within the left half-plane. The strict overlap sectors are
\begin{dmath}\label{eq:hankel-overlap-sectors}
 {\left|\theta-\pi\right|<\pi(1-\rho/2),\quad 0<\rho<2
 \quad\text{for }R,\qquad
 \left|\theta-\pi\right|<\frac\pi2(1-\rho),\quad 0<\rho<1
 \quad\text{for }E}
\end{dmath}.
Appendix~\ref{app:hankel-overlap} supplies the joint estimate. The Hankel construction remains available beyond these ranges; it does not assert convergence of the unchanged vertical integral there. More general profiles require the same inversion and joint-integrability check before the two realizations are identified.

\subsection{Paradigmatic test cases}\label{sec:pairing-test-cases}

\begin{example}[Euler's factorial state]\label[example]{ex:first-MB-pairing}
For $R(z;\zeta)=(1-\zeta z)^{-1}$ and $\varphi(t)=\Gamma(1+t)$, choose $-1<c<0$. The moment measure is $\ee^{-x}\dd x$. On the planar determination $0<\arg_{2\pi}\zeta<2\pi$, \eqref{eq:moment-evaluation} gives
\begin{dmath}\label{eq:MB_Euler}
 \mathcal A_{2\pi}(\zeta)
 ={\int_0^\infty\frac{\ee^{-x}}{1-\zeta x}\,\dd x
 =-\frac{\ee^{-1/\zeta}}{\zeta}
 E_1\left(-\frac1\zeta\right)}
\end{dmath},
where $E_1$ is the exponential integral on its principal
determination, with $E_1(w)=\Gamma(0,w)$
\cite[Eq.~8.4.4]{DLMF}. This determination is consistent with
$0<\arg_{2\pi}\zeta<2\pi$, since
$\Arg(-1/\zeta)=\pi-\arg_{2\pi}\zeta$.
The Barnes integral itself converges on $|\theta-\pi|<3\pi/2$. Thus the real integral gives one planar restriction of a more extensive lifted realization.
A finite geometric identity proves
\begin{dmath}\label{eq:euler-remainder}
 \mathcal A_{2\pi}(\zeta)-\sum_{n=0}^{N-1}n!\zeta^n
 =\zeta^N\int_0^\infty\frac{x^N\ee^{-x}}{1-\zeta x}\dd x
\end{dmath}.
On $|\arg_{2\pi}\zeta-\pi|\leq\pi-\epsilon$, the denominator has a lower bound $C_\epsilon(1+|\zeta|x)$. The remainder is bounded by $C_\epsilon^{-1}N!|\zeta|^N$, proving strong Gevrey-one asymptotics there. The series diverges at every nonzero $\zeta$, but the pairing is an analytic function. Its ordinary coupling Borel transform is $(1-\xi)^{-1}$, and rescaling the Laplace variable identifies the same integral with its Borel sum in directions avoiding the positive real axis. The larger Barnes convergence sector has not been inferred to have this same remainder estimate.
\end{example}

\begin{example}[A variable Gamma slope]\label[example]{ex:exp_kernel_gamma_at}
For the exponential representative and $\varphi(t)=\Gamma(1+at)$, $a>0$, take $-1/a<c<0$. A change of variable in the Gamma moment integral gives, initially for $\RE\zeta<0$,
\begin{dmath}\label{eq:gamma-a-pairing}
 \mathcal A_{2\pi}(\zeta)
 ={\frac1{2\pi i}\int_{\mathcal C_c}
 (-\zeta)^t\Gamma(-t)\Gamma(1+at)\dd t
 =\int_0^\infty\ee^{-x+\zeta x^a}\dd x}
\end{dmath}.
The Barnes integral converges on $|\theta-\pi|<\pi(1+a)/2$. For $0<a<1$, the real integral extends to an entire function and its Taylor coefficients are $\Gamma(1+an)/n!$. For $a=1$, it gives $(1-\zeta)^{-1}$ on $\RE\zeta<1$, with Taylor radius one. For $a>1$, those coefficients have exact Gevrey order $a-1$ and zero convergence radius. In this last case Taylor's integral remainder gives
\begin{dmath}\label{eq:gamma-a-remainder}
 \left|\mathcal A_{2\pi}(\zeta)
 -\sum_{n=0}^{N-1}\frac{\Gamma(1+an)}{n!}\zeta^n\right|
 \leq\frac{\Gamma(1+aN)}{N!}|\zeta|^N,
 \qquad {\RE\zeta\leq0}
\end{dmath},
with boundary values on the imaginary axis understood through the real integral. Stirling's formula converts this into a strong Gevrey-$(a-1)$ bound on the left sector. The transition from an entire value to a divergent expansion comes from the growth of the sampled state; the auxiliary exponential is unchanged.
\end{example}

\begin{example}[Coalescing poles and logarithms]\label[example]{ex:exp_kernel_gamma_minus_t}
For the exponential representative and $\varphi(t)=\Gamma(1-t)$, a line $c<0$ is admissible. Put $q=-\zeta$ with $\arg q=\theta-\pi$. For $q>0$, moment reduction and a classical Bessel integral give \cite[Eq.~10.32.10]{DLMF}
\begin{dmath}\label{eq:coalescing-poles}
 \mathcal A_{2\pi}(-q)
 ={
 \frac1{2\pi i}\int_{\mathcal C_c}
 q^t\Gamma(-t)\Gamma(1-t)\dd t
 ={\int_0^\infty\ee^{-x-q/x}\dd x
 =2\sqrt q\,K_1(2\sqrt q)}
 }
\end{dmath}.
The Barnes integral continues this value to $|\arg q|<\pi$. Its pole at $t=0$ is simple, while those at positive integers are double. Their first contributions are
\begin{dmath}\label{eq:coalescing-leading-terms}
 \mathcal A_{2\pi}(-q)
 ={1+q(\log q+2\gamma-1)
 +O\!\left(|q|^2(1+|\log q|)\right)}
\end{dmath}
as $q\to0$ on proper subsectors, where $\gamma$ is Euler's constant. The double residues account for the logarithm, consistently with the expansion of $K_1$ \cite[Section~10.31]{DLMF}. The value tends to one, but is not holomorphic at zero. 
Applying the ordinary indicial prescription to $\exp(\zeta\uo)[\varphi]$ would require forming
\begin{dmath*}
 \sum_{n=0}^{\infty}
 \frac{\Gamma(1-n)}{n!}\zeta^n
\end{dmath*}.
The constant coefficient is well defined and equals one, but $\Gamma(1-n)$ has a pole for every integer $n\geq1$.
Consequently, this expression does not define even a formal power series with complex coefficients. The obstruction is not divergence, but the absence of finite coefficients.
The Barnes pairing nevertheless exists on the stated domain.
The poles of the ground state at the positive integers coincide with those of $\Gamma(-t)$, producing double poles in the complete spectral density. Their residues yield an expansion involving both powers and logarithms, beginning with \eqref{eq:coalescing-leading-terms}, rather than the ordinary indicial power series.
\end{example}

\subsection{Interpretation of the analytic pairing}\label{sec:interpretation}

The examples suggest a precise route from an analytic pairing to the formal rule. A finite contour displacement first gives an exact identity. Decay in the imaginary direction makes that displacement possible; control as the contour moves to the right determines the remainder. These two estimates have different roles.

For clarity, we first consider jump functions whose dependence
on the external parameter has the form
\begin{dmath*}
 \widetilde \jj(t;r,\theta)=\ee^{tL}h(t),
 \qquad {L=\ln r+i(\theta-\pi)}
\end{dmath*},
where $h(t)$ is independent of $r$ and $\theta$.
The exponential and rational representatives both have this
form, as shown in \eqref{eq:elementary-lifted-jumps}.
For a fixed ground state, write $G(t)=h(t)\varphi(t)$.
The pairing then has integrand $\ee^{tL}G(t)$: the complex
power contains the external-parameter dependence, while
$G$ determines the spectral singularities and decay. This notation is only an abbreviation for the product, not a new ground state.

\begin{theorem}[Recovery of formal umbral evaluation]\label[theorem]{thm:recovery_formal}
Let $F(w)=\sum_{n\geq0}a_nw^n$ be holomorphic near zero and suppose $F(\zeta z)$ has the continued jump just described. Choose $-1<c<0$. Assume $\varphi$ is holomorphic at $\mathbb N$ and $G$ is meromorphic in a neighbourhood of $\RE t\geq c$, with no singularities there other than at most simple poles at $n\in\mathbb N$, whose residues are
\begin{dmath}\label{eq:recovery-residues}
 \Res_{t=n}G(t)=(-1)^{n+1}a_n\varphi(n)
\end{dmath}.
Fix $0<\delta<\min\{-c,1/2\}$. Suppose that, outside the discs $|t-n|<\delta$, there is an $a>0$ such that on every finite strip
\begin{dmath}\label{eq:finite-strip-density-bound}
 {|G(x+iy)|\leq C_X(1+|y|)^{m_X}\ee^{-a|y|},
 \qquad c\leq x\leq X}
\end{dmath}.
Then the pairing is holomorphic on $|\theta-\pi|<a$ and, as $r\to0$, has the ordinary asymptotic expansion
\begin{dmath}\label{eq:recovered-formal-series}
 \widehat{\mathcal A}(\zeta)=\sum_{n=0}^\infty a_n\varphi(n)\zeta^n
\end{dmath}
on every proper closed subsector. If, on the same punctured half-plane, the stronger bound
\begin{dmath}\label{eq:uniform-density-bound}
 |G(x+iy)|\leq C B^{x_+}\Gamma(1+s x_+)
 (1+x_++|y|)^m\ee^{-a|y|},
 \qquad {x_+=\max(x,0),\quad s\geq0,\quad B\geq1}
\end{dmath}
holds with constants independent of $x,y$, the expansion is strong Gevrey-$s$. Independently, convergence of the displaced remainder to zero at fixed $\zeta$ gives equality with the residue series there; for $s=0$ this equality holds near zero on the sector.
\end{theorem}
\begin{proof}
The first bound gives absolute convergence on the initial line. Moving it to $d_N=N+1/2$ crosses the poles $0,\ldots,N$. The horizontal connectors vanish by the same bound on that finite strip. The clockwise orientation gives minus the residues, and $\ee^{nL}=(-1)^n\zeta^n$. Hence
\begin{dmath}\label{eq:finite-recovery}
 {\widetilde{\mathcal A}(r,\theta)
 =\sum_{n=0}^{N}a_n\varphi(n)\zeta^n+E_N(r,\theta),\qquad
 E_N(r,\theta)=\frac1{2\pi i}\int_{d_N+i\mathbb R}\ee^{tL}G(t)\dd t}
\end{dmath}.
For fixed $N$, $E_N=O(r^{N+1/2})$. After subtracting only the first $N$ terms, the term of index $N$ remains in the remainder, proving ordinary asymptotics. The uniform bound controls both this coefficient and $E_N$ by constants times $D^N\Gamma(1+sN)r^N$; Appendix~\ref{app:uniform-recovery} gives the estimates. Finally, $E_N\to0$ gives the stated equality with the infinite residue sum. At order zero, the strong remainder bound gives that limit for sufficiently small $r$.
\end{proof}

The residue condition is the endpoint identity \eqref{eq:primary-spectral-residues}, with the state included. The theorem makes explicit which additional estimates are needed to use it in a pairing \cite{ParisKaminski2001}. Higher-order poles or further singularities change the residue expansion, as the preceding example shows. A germ $F$ and finite samples alone do not supply the hypotheses.

For the parametric representative of \cref{sec:parametric}, the substitution $t=\nu+\mu s$ on its natural oriented contour gives
\begin{dmath}\label{eq:affine-pairing}
 \langle\Delta,\varphi\rangle(\zeta)
 =\frac{\zeta^\beta}{2\pi i}
 \int_{c+i\mathbb R}\jj_F(s;\zeta^\alpha)\,
 \varphi(\nu+\mu s)\dd s
\end{dmath}.
The Jacobian cancels the factor $1/\mu$ in \eqref{eq:affine-mellin}. If the theorem applies to this transformed product, the corresponding formal expression is
\begin{dmath}\label{eq:affine-formal-evaluation}
 \widehat{\mathcal A}(\zeta)
 =\zeta^\beta\sum_{n=0}^\infty
 a_n\varphi(\nu+\mu n)\zeta^{\alpha n}
\end{dmath}.
The remainder is measured in $\eta=\zeta^\alpha\to0$ after dividing out $\zeta^\beta$, on domains whose lifted image stays in the proved $\eta$-sector. For complex $\mu$, the state must be estimated on the tilted contour, not on an unrelated vertical line.

The formal prescription uses only the values of the ground state at nonnegative integers. It is therefore natural to ask whether these values also determine the analytic pairing. The following example shows that they do not, even when the common formal series converges. What is missing is information about the analytic state away from the sampled indices.
\begin{example}[Equal samples and different analytic pairings]\label[example]{ex:equal-samples}
For the rational representative take
\begin{dmath*}
 {\varphi_0(t)=1,\qquad
 \varphi_\lambda(t)=1+\frac{\lambda}{\Gamma(-t)},\qquad \lambda\in\mathbb C}
\end{dmath*}.
Both states equal one at every nonnegative integer. Set $q=-\zeta$ on $|\arg q|<\pi/2$ and choose $-1<c<0$. Cancellation in the difference of the complete densities gives
\begin{dmath}\label{eq:equal-samples-flat-difference}
 \langle R,\varphi_\lambda\rangle_{\mathcal C_c}
 -\langle R,\varphi_0\rangle_{\mathcal C_c}
 ={
 \frac{\lambda}{2\pi i}\int_{\mathcal C_c}
 q^t\Gamma(1+t)\dd t
 =\frac{\lambda}{q}\ee^{-1/q}
 }
\end{dmath}.
The last equality is Mellin inversion for $x\ee^{-x}$ at $x=1/q$, followed by continuation. The unperturbed value is $(1+q)^{-1}$. On each proper subsector the extra term is smaller than every power as $q\to0$, so both pairings have the same, even convergent, formal expansion. For $\lambda\neq0$, the perturbed pairing has this convergent series
only as a sectorial asymptotic expansion. Its exponentially small additional term prevents equality with the series and excludes a strong order-zero remainder estimate. This does not contradict uniqueness for a holomorphic germ, or uniqueness under sufficiently strong sectorial Gevrey conditions. 

\noindent The example distinguishes admissibility of a contour pairing from the stronger conditions under which its value is determined by the formal expansion. It also shows why recovery of formal coefficients and uniqueness of analytic evaluation are separate questions, and why an analytic ground state contains more information than its integer samples.
\end{example}

\paragraph{Identification with Borel--Laplace summation.}
The recovery theorem establishes that the AUF value has the expected
asymptotic expansion. To identify this value with a Borel--Laplace sum,
we must also justify continuation and growth of the Borel transform of
that expansion. A bound on the coefficients only gives its convergence
near the Borel origin. The following result supplies the additional
information directly from the uniform spectral estimate.

\begin{proposition}[Borel--Laplace identification from the spectral estimate]
\label[proposition]{prop:AUF-BL-identification}
Assume the hypotheses of \cref{thm:recovery_formal}, including
\eqref{eq:uniform-density-bound}, with $s>0$ and
\begin{dmath}\label{eq:BL-angular-margin}
 a>\frac{\pi s}{2}
\end{dmath}.
Put $k=1/s$ and write
\begin{dmath*}
 {b_n=a_n\varphi(n),\qquad
 \widehat{\mathcal A}(\zeta)=\sum_{n=0}^{\infty}b_n\zeta^n}
\end{dmath*}.
Then this series is $k$-summable in every lifted direction $d$ satisfying
\begin{dmath}\label{eq:BL-guaranteed-directions}
 |d-\pi|<a-\frac{\pi s}{2}
\end{dmath}.
Its $k$-Borel transform
\begin{dmath}\label{eq:AUF-coupling-Borel}
 B_k(\xi)=\sum_{n=0}^{\infty}
 \frac{b_n\xi^n}{\Gamma(1+sn)}
\end{dmath}
converges near zero and admits analytic continuation to an unbounded
sector about each such direction, with at most algebraic growth there.
The corresponding Borel--Laplace sum, in the normalization of
\eqref{eq:k-Laplace-normalisation}, satisfies
\begin{dmath}\label{eq:AUF-BL-equality}
 (\mathcal S_{k,d}\widehat{\mathcal A})(\zeta)
 =\widetilde{\mathcal A}(r,\theta),
 \qquad {\zeta=r\ee^{i\theta},\quad
 |\theta-d|<\frac{\pi}{2k}}
\end{dmath},
with the indicated logarithmic lifts. Equality is initially a statement
near the origin and persists along matched analytic continuations.
\end{proposition}

\begin{proof}
The proof constructs the Borel transform and then shows that Laplace
integration recovers the original pairing. Choose
$\max\{c,-1/s\}<c_1<0$. Moving the original contour to
$\mathcal C_{c_1}=c_1+i\mathbb R$ crosses no poles, and the joining
integrals vanish by the finite-strip estimate. This change does not
alter the AUF value. It ensures $1+sc_1>0$, which will be needed for a
Gamma integral.

\emph{Construction of the Borel transform.}
For $\xi=\rho\ee^{i\vartheta}$ on the covering, define
\begin{dmath}\label{eq:Borel-Barnes-representation}
 \widetilde B_k(\rho,\vartheta)
 =\frac1{2\pi i}\int_{\mathcal C_{c_1}}
 \ee^{t[\ln\rho+i(\vartheta-\pi)]}
 \frac{G(t)}{\Gamma(1+st)}\,\dd t
\end{dmath}.
Division by $\Gamma(1+st)$ divides the residue at each nonnegative
integer $n$ by $\Gamma(1+sn)$, exactly as required in
\eqref{eq:AUF-coupling-Borel}. There is also an analytic cost: the
reciprocal Gamma factor grows vertically at exponential rate
$\pi s/2$. The remaining decay is therefore positive on
\begin{dmath}\label{eq:Borel-lifted-sector}
 |\vartheta-\pi|<a-\frac{\pi s}{2}
\end{dmath}.
Appendix~\ref{app:BL-identification} proves that this integral
converges locally uniformly and equals the Borel series near
zero. It also establishes at most algebraic growth on proper
closed angular subsectors. Thus the integral supplies the
analytic continuation and growth bounds required for
$k$-summability, not merely the correct formal residues.

\emph{Recovery by Laplace integration.}
Fix a direction $d$ in \eqref{eq:BL-guaranteed-directions} and first take
$\zeta=r\ee^{id}$. The substitution $\xi=\zeta u^s$, $u>0$, in
\eqref{eq:k-Laplace-normalisation} gives
\begin{dmath}\label{eq:BL-positive-u}
 (\mathcal S_{k,d}\widehat{\mathcal A})(\zeta)
 =\int_0^\infty\ee^{-u}
 \widetilde B_k(ru^s,d)\,\dd u
\end{dmath}.
Let $L=\ln r+i(d-\pi)$. Substitution of
\eqref{eq:Borel-Barnes-representation} produces an absolutely
convergent double integral: its absolute value before the factor
$1/(2\pi)$ is bounded by
\begin{dmath*}
 \Gamma(1+sc_1)
 \int_{\mathcal C_{c_1}}
 \left|\ee^{tL}\frac{G(t)}{\Gamma(1+st)}\right|\,|\dd t|
 <\infty
\end{dmath*}.
Fubini's theorem therefore allows the order of integration to be
reversed. Euler's Gamma integral then gives
\begin{dmath}\label{eq:BL-restores-Gamma}
 (\mathcal S_{k,d}\widehat{\mathcal A})(\zeta)
 ={
 \frac1{2\pi i}\int_{\mathcal C_{c_1}}
 \ee^{tL}\frac{G(t)}{\Gamma(1+st)}
 \left(\int_0^\infty\ee^{-u}u^{st}\,\dd u\right)\dd t
 =\frac1{2\pi i}\int_{\mathcal C_{c_1}}\ee^{tL}G(t)\,\dd t
 =\widetilde{\mathcal A}(r,d)
 }
\end{dmath}.
Laplace integration has restored exactly the Gamma factor removed in
the Borel construction. For a fixed ray $d$, the Laplace integral is
holomorphic on $|\theta-d|<\pi/(2k)$; this sector lies inside
$|\theta-\pi|<a$ by \eqref{eq:BL-guaranteed-directions}. The identity
theorem extends the equality from the ray to this sector. Varying $d$ within the open Borel sector glues these Laplace integrals on overlaps. Around each permitted direction, a sufficiently small variation gives a common sector of opening greater than $\pi/k$; further equality holds along matched analytic continuations.
\end{proof}

This argument proves both existence of the Borel--Laplace sum and its
identification with the AUF value under explicit sufficient conditions.
The two roles of \eqref{eq:BL-angular-margin} are consistent: it leaves
a nonempty sector for the continued Borel transform, and it gives an
AUF sector of opening greater than $\pi s$. In the latter sector,
Watson uniqueness excludes a nonzero difference with zero strong
Gevrey-$s$ expansion \cite{Balser1994,mLR16}. For an independently
constructed Borel--Laplace sum, this gives a second comparison argument
whenever both strong expansions hold on the same sector of that
opening. A smaller common domain of holomorphy alone is not sufficient.

The Borel transform in \eqref{eq:AUF-coupling-Borel} concerns the
\emph{evaluated} series in the external parameter. It includes the
ground-state weights and is not the auxiliary representative
$\Delta(\ee^u;\zeta)$. At $s=0$, no positive-order summation is needed:
the strong order-zero remainder in \cref{thm:recovery_formal} already
proves equality with the convergent series. Without the uniform bound
and the angular margin used above, the pairing may still exist, but
this proposition makes no summability assertion.

\subsection{Spectral transmutation law}\label{sec:spectral_transmutation}

Spectral duality now has an immediate consequence for evaluation. A Gamma factor can be transferred between the jump function and the ground state without changing their product. This relates two analytic realizations before any residue expansion is taken.

\begin{corollary}[Pairing-preserving transmutation]\label[corollary]{cor:reflection_pairings}
Let the representatives $E$ and $R_E$ satisfy $\jj_{R_E}(t;\zeta)=\Gamma(1+t)\,\jj_E(t;\zeta)$ on a common transform domain and matched continuation. On the same contour and logarithmic lift, whenever the complete density is admissible,
\begin{dmath}\label{eq:transmutation}
 \langle E,\varphi\rangle_{\mathcal C}
 =\left\langle R_E,\frac{\varphi}{\Gamma(1+\cdot)}\right\rangle_{\mathcal C}
\end{dmath}.
For the parametric representatives, the divisor is $\Gamma(1+(t-\nu)/\mu)$.
\end{corollary}
\begin{proof}
The complete densities are identical by \eqref{eq:general-gamma-duality} or its parametric version \eqref{eq:parametric-spectral-duality}. Their convergent contour integrals therefore agree. The identity extends along matched analytic continuations. Removable cancellations are interpreted in the complete product.
\end{proof}

Once the spectral identity has been established in \cref{sec:spectral_duality}, no additional interchange of P{\'o}lya and Barnes integrals is needed. The coefficients also show what is preserved: $E(w)=\sum a_nw^n$ and $R_E(w)=\sum n!a_nw^n$ give the same terms
\begin{dmath*}
 a_n\varphi(n)\zeta^n
 =n!a_n\frac{\varphi(n)}{\Gamma(1+n)}\zeta^n
\end{dmath*}.
The recovery theorem identifies these formal expressions with the appropriate expansions of the equal analytic values. Transmutation does not lower the Gevrey order of that unchanged series or improve the decay of its unchanged complete density.
An operation performed after transmutation must satisfy its own analytic conditions; equality of the original pairings does not automatically justify that operation.

\section{Worked examples and analytic realisations}\label{sec:examples}

In this section we apply the framework to a number of selected illustrative examples.
The examples use positive-axis moment integrals, Hankel contours, direct Barnes displacement and deformed ground states. Each specifies the representative, state and initial domain, distinguishes continuation from integral convergence, and supplies a remainder estimate for a claimed asymptotic expansion.

\subsection{Elementary Mellin kernels}\label{sec:elementary-mellin}

The following ground states admit particularly simple moment measures:
\begin{dgroup}\label{eq:elementary-moment-measures}
\begin{dmath*}
 {\lambda(t)=\frac{\Gamma(1+t)}{\Gamma(1+t/2)},\qquad
 \dd\omega_G(x)=\frac{\ee^{-x^2/4}}{\sqrt\pi}\dd x,\quad x>0}
\end{dmath*},
\begin{dmath*}
 {\varphi_b(t)=\frac{\Gamma(1+t)}{\Gamma(b+t)},\qquad
 \dd\omega_b(x)=\frac{(1-x)^{b-2}}{\Gamma(b-1)}\dd x,
 \quad 0<x<1,\quad \RE b>1}
\end{dmath*}.
\end{dgroup}
Each state is the integral of $x^t$ against the indicated measure for $\RE t>-1$, by the Gamma duplication formula and the beta integral \cite[Sections~5.5 and~5.12]{DLMF}. In the convention of \eqref{eq:mellin-reduction}, $W_\varphi(x)$ is $x$ times the displayed density. The moment criterion \eqref{eq:moment-evaluation} identifies the auxiliary integrals with Barnes pairings on their common initial domains. Their subsequent continuation is addressed in each example.

\begin{example}[The Gaussian ground state]\label[example]{ex:gaussian_kernel}
For $\lambda$ in \eqref{eq:elementary-moment-measures}, the representative $\ee^{i\zeta z}$ gives, initially for $\IM\zeta>0$,
\begin{dmath}\label{eq:gaussian-entire}
 e_G(\zeta)\defeq {
 \langle\ee^{i\zeta z},\lambda\rangle
 =\frac1{\sqrt\pi}\int_0^\infty
 \ee^{-x^2/4+i\zeta x}\dd x
 =\ee^{-\zeta^2}\operatorname{erfc}(-i\zeta)
 }
\end{dmath}.
Gaussian domination on compact $\zeta$-sets proves entire continuation and permits termwise expansion. The series and its even and odd parts are
\begin{dmath}\label{eq:gaussian-trigonometry}
 {e_G(\zeta)=\sum_{n=0}^\infty\frac{(i\zeta)^n}{\Gamma(1+n/2)},\qquad
 \cos_G\zeta=\ee^{-\zeta^2},\qquad
 \sin_G\zeta=\ee^{-\zeta^2}\operatorname{erfi}(\zeta)}
\end{dmath}.
These are the Gaussian trigonometric functions considered in \cite{gD23,Ricci2026}. Their cosine and sine evaluations follow by linearity after continuation; opposite exponentials need not share an initial decay ray.

The same state gives divergent expansions for rational representatives. For $p=1,2$, write their planar values on $\mathbb C\setminus[0,\infty)$ as
\begin{dmath}\label{eq:gaussian-rational}
 G_p(\zeta)\defeq{
 \left\langle\frac1{1-\zeta z^p},\lambda\right\rangle
 =\frac1{\sqrt\pi}\int_0^\infty
 \frac{\ee^{-x^2/4}}{1-\zeta x^p}\dd x
 }
\end{dmath},
with $0<\arg_{2\pi}\zeta<2\pi$. The affine spectral rule gives a Barnes realization on $|\theta-\pi|<\pi(1+p/4)$, with a line $-\min\{1,1/p\}<\RE s<0$ in the variable $s=t/p$. For the real-integral determination, a finite geometric identity yields
\begin{dmath}\label{eq:gaussian-finite-remainder}
 G_p(\zeta)-\sum_{n=0}^{N-1}\lambda(pn)\zeta^n
 =\frac{\zeta^N}{\sqrt\pi}\int_0^\infty
 \frac{x^{pN}\ee^{-x^2/4}}{1-\zeta x^p}\dd x
\end{dmath}.
On $|\arg_{2\pi}\zeta-\pi|\leq\pi-\epsilon$, the modulus of the denominator is at least $\sin(\epsilon/2)(1+|\zeta|x^p)$. The remainder is consequently bounded by $\csc(\epsilon/2)\lambda(pN)|\zeta|^N$. Stirling's formula gives strong Gevrey order $p/2$ on these subsectors; the coefficient growth excludes every smaller order. The larger Barnes convergence sector is a separate statement.

For the quadratic representative the value is also
\begin{dmath}\label{eq:quadratic-gaussian-value}
 G_2(\zeta)=\frac{\sqrt\pi}{2\sqrt{-\zeta}}
 \ee^{-1/(4\zeta)}
 \operatorname{erfc}\left(\frac1{2\sqrt{-\zeta}}\right)
\end{dmath},
first for $\zeta<0$ and then with the square root fixed by $\Log(-\zeta)$. This follows by inserting the elementary Laplace representation of the denominator and evaluating the Gaussian integral. The change of Gevrey order between $G_1$ and $G_2$ comes from sampling $\lambda$ at $n$ or $2n$, not from changing the state.
\end{example}

\begin{example}[The Beta-type ground state]\label[example]{ex:beta_kernel}
For $\RE b>1$, the compactly supported measure gives
\begin{dgroup}\label{eq:beta-pairings}
\begin{dmath*}
 \langle E,\varphi_b\rangle
 ={\frac1{\Gamma(b-1)}\int_0^1
 \ee^{\zeta x}(1-x)^{b-2}\dd x
 =\frac{{}_1F_1(1;b;\zeta)}{\Gamma(b)}}
\end{dmath*},
\begin{dmath*}
 \langle R,\varphi_b\rangle
 ={\frac1{\Gamma(b-1)}\int_0^1
 \frac{(1-x)^{b-2}}{1-\zeta x}\dd x
 =\frac{{}_2F_1(1,1;b;\zeta)}{\Gamma(b)}}
\end{dmath*}.
\end{dgroup}
The equalities initially follow on the elementary transform domains and extend by the displayed integrals \cite[Sections~13.4 and~15.6]{DLMF}. The first value is entire, with series $\sum\zeta^n/\Gamma(b+n)$. The second has Taylor series $\sum n!\zeta^n/\Gamma(b+n)$ of radius one, and the integral continues it to $\mathbb C\setminus[1,\infty)$. The cut begins when the pole $x=1/\zeta$ reaches the endpoint $x=1$ and then enters $(0,1)$. At $b=2$ the completions reduce to $(\ee^\zeta-1)/\zeta$ and $-\log(1-\zeta)/\zeta$, with removable values at zero. Continuing these special functions in $b$ does not remove the integrability condition on the original measure.
\end{example}

\begin{example}[Umbral evaluation of a Gaussian Mellin integral]\label[example]{ex:gaussian-operational}
We compute the improper integral
$I(s)\defeq\int_0^\infty\zeta^{s-1}\ee^{-\zeta^2}\dd\zeta$,
for $\RE s>0$. The strategy is to represent the Gaussian by an
auxiliary exponential, integrate in $\zeta$, and then evaluate
the resulting negative power.

For this purpose, take the entire ground state
\begin{dmath}\label{eq:gaussian-operational-state}
 \chi(t)\defeq\ee^{i\pi t/2}\frac{2^t\sqrt\pi}{\Gamma((1-t)/2)}
\end{dmath}.
Its moments are realized on $\mathbb R+ih$, with $h>0$ and
$0<\arg z<\pi$, by
\begin{dmath}\label{eq:gaussian-contour-moments}
 \chi(t)=\frac1{2\sqrt\pi}\int_{\mathbb R+ih}
 \ee^{-z^2/4}z^t\dd z,\qquad {t\in\mathbb C}
\end{dmath}.
The same contour represents the Gaussian as the AUF pairing
\begin{dmath}\label{eq:gaussian-contour-evaluation}
 \langle\ee^{i\zeta z},\chi\rangle
 ={\frac1{2\sqrt\pi}\int_{\mathbb R+ih}
 \ee^{-z^2/4+i\zeta z}\dd z=\ee^{-\zeta^2}},\qquad {\zeta>0}
\end{dmath}.
Appendix~\ref{app:gaussian-operational-proof} proves the moment
identity and agreement with the Barnes realization.

Substituting this representation into $I(s)$ gives a double
integral. The contour shift supplies the damping
$|\ee^{i\zeta(x+ih)}|=\ee^{-h\zeta}$. Indeed, with
$\sigma=\RE s>0$,
\begin{dmath}\label{eq:gaussian-joint-integrability}
 \frac1{2\sqrt\pi}\int_{\mathbb R}\int_0^\infty
 \left|\ee^{-(x+ih)^2/4+i\zeta(x+ih)}\zeta^{s-1}\right|
 \dd\zeta\dd x
 =\ee^{h^2/4}\Gamma(\sigma)h^{-\sigma}
\end{dmath}.
Fubini's theorem therefore permits integration in $\zeta$
first. For fixed $z$ on the contour, the elementary Laplace
integral is
$\int_0^\infty\zeta^{s-1}\ee^{i\zeta z}\dd\zeta
=\Gamma(s)(-iz)^{-s}$.
Since $\log(-iz)=\log z-i\pi/2$, the remaining contour
integral is precisely the moment $\chi(-s)$. Hence
\begin{dmath}\label{eq:gaussian-operational-integral}
 I(s)
 ={\Gamma(s)\ee^{i\pi s/2}\chi(-s)
 =\frac12\Gamma(s/2)},\qquad {\RE s>0}
\end{dmath},
where the last equality follows from Gamma duplication.
Termwise integration of the Gaussian's Taylor series would
instead give divergent integrals. Here the exact auxiliary
exponential is integrated first, and the resulting negative
power is evaluated through a convergent contour moment.
\end{example}

\subsection{The inverse Gamma ground state and its Hankel realisation}\label{sec:inverse-gamma-examples}

We apply \cref{thm:hankel} to $1/\Gamma(1+t)$, keeping the reciprocal-Gamma contour \cite{WhittakerWatson1927} at a radius large enough to avoid the substituted representative's singularities.

\begin{example}[Inverse-Gamma evaluations]\label[example]{ex:inverse_gamma_hankel}
More generally, let $\varphi(t)=1/\Gamma(b+\rho t)$, where $\rho>0$ and $b\in\mathbb C$. Applying \cref{thm:hankel} to the elementary representatives gives
\begin{dgroup}\label{eq:inverse-gamma-entire-pairings}
\begin{dmath*}
 \langle R,\varphi\rangle_H
 ={\sum_{n=0}^\infty\frac{\zeta^n}{\Gamma(b+\rho n)}
 =E_{\rho,b}(\zeta)}
\end{dmath*},
\begin{dmath*}
 \langle E,\varphi\rangle_H
 ={\sum_{n=0}^\infty\frac{\zeta^n}{n!\Gamma(b+\rho n)}
 =W_{\rho,b}(\zeta)}
\end{dmath*}.
\end{dgroup}
Here the subscript $H$ identifies the Hankel realization. The entire series are the Mittag--Leffler and Wright functions \cite[Section~10.46]{DLMF}, already related at the profile level in \cref{sec:spectral_duality}. The theorem now supplies their analytic ground-state evaluations. For $b=\rho=1$, the values are $\ee^\zeta$ and $I_0(2\sqrt\zeta)$, respectively. The latter notation means the entire series $\sum\zeta^n/(n!)^2$; it has no square-root monodromy. The rational vertical pairing agrees on $|\theta-\pi|<\pi/2$, whereas the exponential one has no open absolute-convergence sector on the unchanged vertical contour. The Hankel theorem, not that vertical integral, establishes its value.

The fractional auxiliary power of the original Gaussian construction also fits directly. With $\varphi(t)=1/\Gamma(1+t)$,
\begin{dmath}\label{eq:fractional-gaussian-hankel}
 \left\langle\frac1{1-i\zeta z^{1/2}},\varphi\right\rangle_H
 ={\sum_{n=0}^\infty\frac{(i\zeta)^n}{\Gamma(1+n/2)}
 =e_G(\zeta)}
\end{dmath}.
Choose the Hankel radius larger than $|\zeta|^2$ on the compact set being considered, and use $\log z=-\log w$. The even part of the representative and its odd part divided by $i$
are $(1+\zeta^2z)^{-1}$ and
$\zeta z^{1/2}(1+\zeta^2z)^{-1}$; their values are
$\cos_G\zeta$ and $\sin_G\zeta$ from
\eqref{eq:gaussian-trigonometry}. No deformation through the possible pole on the negative Hankel cut is needed. Equivalently, the affine spectral rule changes the state to $1/\Gamma(1+t/2)$ and transmutation then gives $\lambda(t)$. This proves agreement with the Gaussian exponential pairing on a common vertical domain and then by entire continuation.
\end{example}

\subsection{Genuinely Barnes pairings}\label{sec:genuine-barnes}

A Barnes integral can define the value without an elementary moment measure. The next example obtains expansions at both endpoints and shows the effect of pole orders and cancellations.

\begin{example}[A filtered logarithmic expansion]\label[example]{ex:a-third}
Take the rational representative and $\varphi(t)=\Gamma(1+t)/\Gamma(1+t/3)$. With $-1<c<0$ and $L=\ln r+i(\theta-\pi)$, the pairing is
\begin{dmath}\label{eq:third-barnes}
 \widetilde{\mathcal A}(r,\theta)
 =\frac1{2\pi i}\int_{c+i\mathbb R}
 \ee^{tL}\frac{\Gamma(-t)\Gamma(1+t)^2}{\Gamma(1+t/3)}\dd t
\end{dmath}.
The complete density decays vertically with exponential rate $4\pi/3$. Thus the integral is holomorphic on $|\theta-\pi|<4\pi/3$, a sector on the covering rather than one cut plane. Its primary expansion is
\begin{dmath}\label{eq:third-primary}
 \widehat{\mathcal A}(\zeta)
 =\sum_{n=0}^\infty\frac{n!}{\Gamma(1+n/3)}\zeta^n
\end{dmath}.
The coefficients have exact Gevrey order $2/3$. Gamma triplication and a rotated Euler estimate give the uniform bound required by \cref{thm:recovery_formal}, proving a strong Gevrey-$2/3$ remainder on every proper subsector of the full strict convergence sector. Appendix~\ref{app:third-estimate} proves the required estimate.

Leftward displacement crosses double poles at $t=-n$, $n\geq1$, except at multiples of three, where a zero of $1/\Gamma(1+t/3)$ removes one order. Write $B_n(L)$ for the resulting contribution. For $3\nmid n$,
\begin{dmath}\label{eq:third-log-residue}
 B_n(L)=\frac{\ee^{-nL}}{(n-1)!\Gamma(1-n/3)}
 \left[L+\psi(n)-\frac13\psi(1-n/3)\right]
\end{dmath},
where $\psi=\Gamma'/\Gamma$. At $n=3m$ the contribution is instead
\begin{dmath}\label{eq:third-simple-residue}
 B_{3m}(L)=\frac{(-1)^{m-1}(m-1)!}{3(3m-1)!}\ee^{-3mL}
\end{dmath}.
Appendix~\ref{app:third-residues} gives the Laurent calculation. Shifting to $\RE t=-N-\delta$, with $0<\delta<1$, gives a finite sum $\sum_{n=1}^N B_n(L)$ and a remainder $O(r^{-N-\delta})$ as $r\to\infty$ on proper subsectors. Horizontal connectors vanish by fixed-strip Stirling estimates \cite{ParisKaminski2001}. This proves the inverse-power and logarithmic expansion without requiring convergence of an infinite left residue sum.

\begin{figure}[htbp]
 \centering
 \includegraphics[width=0.87\linewidth]{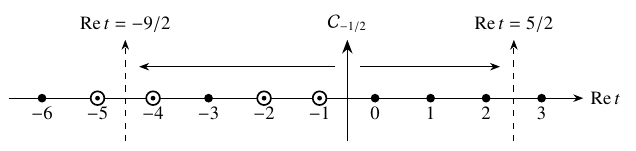}
 \caption{Poles and finite contour displacements for \eqref{eq:third-barnes}. A dot denotes a simple pole; a circled dot denotes a double pole. All vertical lines are oriented upward. Rightward displacement contributes minus the crossed residues; leftward displacement contributes plus the crossed residues.}\label{fig:mellin_barnes_example}
\end{figure}

The contribution at $t=-3$ is $\ee^{-3L}/6$, without a logarithm, unlike those at $-2$ and $-4$. These contour residues replace the undefined indicial evaluation at singular negative integers.
\end{example}

\subsection{Deformations of Gamma-type ground states}\label{sec:deformed-gamma}

Real Gamma powers and algebraic factors give further examples, with branches fixed on the initial spectral domain. The final state lies outside the Gamma family.

\begin{example}[The Le Roy deformation]\label[example]{ex:mu-family}
Let $\kappa\in\mathbb R$ and put $\varphi_\kappa(t)=\exp[-\kappa\log\Gamma(1+t)]$ on $\RE t>-1$, with the logarithm real on the positive real axis. For the rational representative,
\begin{dmath}\label{eq:leroy-pairing}
 \widetilde{\mathcal A}(r,\theta)
 =\frac1{2\pi i}\int_{c+i\mathbb R}
 \ee^{tL}\Gamma(-t)\Gamma(1+t)^{1-\kappa}\dd t,
 \qquad {-1<c<0}
\end{dmath}.
Its strict vertical-convergence sector is $|\theta-\pi|<\pi(1-\kappa/2)$, which is nonempty for $\kappa<2$. The primary coefficients are $1/(n!)^\kappa$. For $0<\kappa<2$, the displaced remainder tends to zero at every fixed parameter in this sector, so the pairing equals the entire series $\sum\zeta^n/(n!)^\kappa$. Appendix~\ref{app:leroy-estimate} proves that limit. At $\kappa=0$ the residue series is geometric and the value continues as $(1-\zeta)^{-1}$. For $\kappa<0$, the same appendix proves strong Gevrey order $-\kappa$ on every proper subsector of the stated sector; Stirling's formula shows that this order is exact.

At $\kappa=2$ the open vertical sector disappears, although the entire series still equals $I_0(2\sqrt\zeta)$. For $\kappa>2$ the series remains entire but is not represented by the unchanged vertical contour. Noninteger $\kappa$ produces branch points upon leftward continuation past the negative integers, not within the initial half-plane. Complex $\kappa$ would require a different growth analysis: its imaginary part couples to the $y\log|y|$ term of $\log\Gamma(c+iy)$. The present sector statement is for real $\kappa$.
\end{example}

\begin{example}[The Lerch transcendent in two realizations]\label[example]{ex:lerch-transmutation}\label[example]{ex:lerch-pairing}
For $a>0$ and $\RE\sigma>0$, take $\psi(t)=(t+a)^{-\sigma}$ with its cut to the left of $-a$. Its moment measure is
\begin{dmath*}
 \dd\omega(x)=\frac{x^{a-1}(-\log x)^{\sigma-1}}{\Gamma(\sigma)}\dd x,
 \qquad {0<x<1}
\end{dmath*}.
With $-\min\{1,a\}<c<0$, the rational pairing reduces initially on $0<\arg_{2\pi}\zeta<2\pi$ to
\begin{dmath}\label{eq:lerch-realization}
 \Phi(\zeta,\sigma,a)
 ={\frac1{\Gamma(\sigma)}\int_0^1
 \frac{x^{a-1}(-\log x)^{\sigma-1}}{1-\zeta x}\dd x
 =\langle R,\psi\rangle
 =\langle E,\Gamma(1+\cdot)\psi\rangle}
\end{dmath}.
The last equality follows by transmutation on the common contour. The integral continues the same value to $\mathbb C\setminus[1,\infty)$ and, for $|\zeta|<1$, gives $\sum_{n\geq0}\zeta^n/(n+a)^\sigma$, the Lerch transcendent \cite{Lerch1887,DLMF}. Dominated expansion proves the series identity. The cut is generated when the moving auxiliary pole meets $(0,1)$, while a noninteger $\sigma$ also gives a branch point at $t=-a$ in the spectral state. These are different analytic features; the external cut is present even for positive integer $\sigma$.
\end{example}

\begin{example}[An expansion beyond finite Gevrey orders]\label[example]{ex:lognormal-state}
For $\zeta=-q<0$, the pairing below is the Laplace transform
$\mathbb E[\ee^{-qX}]$ of a lognormal random variable
$X=\ee^Y$, with $Y$ standard normal. This transform admits
a known Mellin--Barnes representation \cite{Miles2022};
here it illustrates how AUF retains an analytic realization
when the moment expansion lies outside every finite Gevrey
class. Take
\begin{dmath}\label{eq:lognormal-state}
 {\varphi(t)=\ee^{t^2/2},\qquad
 \dd\omega(x)=\frac1{\sqrt{2\pi}}
 \ee^{-(\log x)^2/2}\frac{\dd x}{x},\quad x>0}
\end{dmath}.
The substitution $x=\ee^v$ gives $\int x^t\dd\omega(x)=\ee^{t^2/2}$ for every complex $t$. For the exponential representative, the moment integral $\int_0^\infty\ee^{\zeta x}\dd\omega(x)$ agrees initially on $\RE\zeta<0$ with
\begin{dmath}\label{eq:lognormal-barnes}
 \widetilde{\mathcal A}(r,\theta)
 =\frac1{2\pi i}\int_{c+i\mathbb R}
 \ee^{tL}\Gamma(-t)\ee^{t^2/2}\dd t,
 \qquad {-1<c<0}
\end{dmath}.
Gaussian decay in $\IM t$ controls every finite interval of lifted arguments, so this Barnes integral is holomorphic throughout the logarithmic covering. On each finite spectral strip it satisfies \eqref{eq:finite-strip-density-bound} with any prescribed $a>0$. The recovery theorem therefore gives the ordinary expansion
\begin{dmath}\label{eq:lognormal-expansion}
 \widehat{\mathcal A}(\zeta)
 =\sum_{n=0}^\infty\frac{\ee^{n^2/2}}{n!}\zeta^n
\end{dmath}
on every proper finite angular subsector. On $\RE\zeta\leq0$, Taylor's integral remainder also gives the explicit bound $\ee^{N^2/2}|\zeta|^N/N!$. The coefficient growth exceeds every finite Gevrey bound. This gives ordinary asymptotics beyond all finite-order Borel--Laplace classes, for the specified state and contour rather than from the coefficients alone.
\end{example}

\section{The zero-dimensional quartic model}\label{sec:numerical-illustration}

The zero-dimensional quartic model allows direct comparison of the original integral, its expansions and its analytic continuation. Its normalized partition function is
\begin{dmath}\label{eq:quartic-partition}
 Z(g)=\frac1{\sqrt{2\pi}}\int_{-\infty}^{\infty}
 \exp\left(-\frac{x^2}{2}-\frac{g x^4}{4}\right)\dd x,
 \qquad {\RE g>0}
\end{dmath}.
This classical zero-dimensional field-theory model has an exact special-function solution and is a useful test of Borel summation and nonperturbative reconstruction \cite{MalbouissonPortugalSvaiter2001,Rivasseau2009}. Related quartic models also provide tests in non-Borel-summable settings \cite{Marucho2008}. We derive both residue expansions and the lateral discontinuity from one AUF pairing.

\paragraph{The representative, ground state and lifted value.}
The formal expansion of the quartic exponential gives
\begin{dmath}\label{eq:0d_quartic_mod_div_series}\label{eq:quartic-weak-series}
 {\widehat Z(g)=\sum_{n=0}^\infty a_ng^n,\qquad
 a_n=\frac{(-1)^n\Gamma(2n+1/2)}{\sqrt\pi\,n!}}
\end{dmath}.
The coefficient magnitudes are the normalized Gaussian moments of $(x^4/4)^n$, divided by $n!$. Accordingly, take the representative $E(z;-g)=\ee^{-gz}$ and the state
\begin{dmath}\label{eq:quartic-state}
 {\varphi(t)=\frac{\Gamma(2t+1/2)}{\sqrt\pi}
 =\int_0^\infty z^t\dd\omega(z),\qquad
 \dd\omega(z)=\frac{\ee^{-\sqrt z}}{2\sqrt\pi\,z^{3/4}}\dd z,
 \quad \RE t>-1/4}
\end{dmath}.
The substitution $z=x^4/4$, combining the two equal half-axis contributions, gives this moment measure and analytic state.

Write $g=r\ee^{i\psi}$ on the logarithmic covering, starting at $g>0$ with $\psi=0$. Since the AUF parameter is $\zeta=-g$, its lifted argument is $\theta=\psi+\pi$ and $\log(-\zeta)=\ln r+i\psi$. The pairing is therefore
\begin{dmath}\label{eq:quartic-barnes}
 \widetilde Z(r,\psi)\defeq
 \frac1{2\pi i}\int_{\mathcal C_c}
 \ee^{t(\ln r+i\psi)}\Gamma(-t)
 \frac{\Gamma(2t+1/2)}{\sqrt\pi}\,\dd t,
 \qquad {-1/4<c<0}
\end{dmath},
where $\mathcal C_c=c+i\mathbb R$ is oriented upward. Stirling's formula gives vertical decay $\ee^{-3\pi|\IM t|/2}$, up to a power factor. Thus \eqref{eq:quartic-barnes} converges locally uniformly and defines a holomorphic function for $|\psi|<3\pi/2$.

On $|\psi|<\pi/2$, the moment criterion \eqref{eq:moment-evaluation} applies: the absolute moment in \eqref{eq:quartic-state} and the vertical integral of the exponential jump are both finite. Mellin inversion consequently gives $\widetilde Z=\int_0^\infty\ee^{-gz}\dd\omega(z)=Z(g)$. The Barnes pairing therefore continues the original integral. Its principal planar determination is
\begin{dmath}\label{eq:quartic-planar}
 {Z_\pi(g)\defeq\widetilde Z(|g|,\arg_\pi g),\qquad
 -\pi<\arg_\pi g<\pi}
\end{dmath}.
The larger lifted convergence sector will allow both banks of the negative real axis to be evaluated by the same Barnes contour.

\paragraph{Weak coupling and the remainder.}
Rightward finite displacements cross the poles of $\Gamma(-t)$ and recover \eqref{eq:quartic-weak-series}. Stirling's formula gives $|a_n|\sim4^n n!/(\pi\sqrt2\,n)$, so the series has zero convergence radius and exact Gevrey order one. An estimate of the actual remainder follows directly from Taylor's integral remainder:
\begin{dmath}\label{eq:quartic-weak-remainder}
 \left|Z(g)-\sum_{n=0}^{N-1}a_ng^n\right|
 \leq\frac{\Gamma(2N+1/2)}{\sqrt\pi\,N!}g^N,
 \qquad {g>0}
\end{dmath}.
Only a finite Taylor expansion has been integrated; no interchange with an infinite divergent series is involved.

The estimate also extends to proper subsectors of the lifted Barnes sector. Rotate the original contour to $x=\ee^{i\chi}y$, with $|\chi|<\pi/4$ and $|\psi+4\chi|<\pi/2$. The Gaussian decays, and the quartic exponential has modulus at most one. The remainder bound becomes
\begin{dmath}\label{eq:quartic-sectorial-remainder}
 \left|\widetilde Z(r,\psi)-\sum_{n=0}^{N-1}a_ng^n\right|
 \leq\frac{\Gamma(2N+1/2)r^N}
 {\sqrt\pi\,N!(\cos2\chi)^{2N+1/2}}
\end{dmath}.
The coefficients are unchanged by the rotation. Each fixed $\chi$ defines a holomorphic integral on the indicated angular interval; on overlaps, the connecting arcs vanish and the integrals agree. These intervals cover $|\psi|<3\pi/2$, and every proper closed subsector admits a finite subcover with $\cos2\chi$ bounded away from zero. This proves strong Gevrey-one asymptotics on the lifted sector.

\paragraph{Comparison with the coupling Borel transform.}
Using \eqref{eq:k-Borel-normalisation} with $k=1$, we apply
the original transform \eqref{eq:Borel_transform} to
$t^{-1}\widehat Z(t^{-1})=\sum_{n\geq0}a_nt^{-n-1}$,
which has no constant term. The resulting coupling
Borel transform is
\begin{dmath}\label{eq:quartic-borel}
 B(\xi)\defeq{
 \sum_{n=0}^\infty\frac{a_n}{n!}\xi^n
 ={}_2F_1\left(\frac14,\frac34;1;-4\xi\right)
 }
\end{dmath},
initially for $|\xi|<1/4$. Its continuation from zero has a logarithmic branch point at $\xi=-1/4$ and a principal cut $(-\infty,-1/4]$. Euler's hypergeometric integral bounds it on proper subsectors away from this cut \cite[Sections~15.6 and~15.8]{DLMF}. The corresponding Laplace integral satisfies
\begin{dmath}\label{eq:quartic-borel-sum}
 Z_\pi(g)={
 \int_0^\infty\ee^{-s}B(gs)\,\dd s,
 \qquad {|\arg_\pi g|<\pi}
 }
\end{dmath}.
To verify the equality, take $g>0$ and use the hypergeometric Barnes representation on $-1/4<c<0$. Its density, excluding $(gs)^t$, is $\Gamma(-t)\varphi(t)/\Gamma(1+t)$, by duplication \cite[Sections~5.5 and~15.6]{DLMF}. It decays as $\ee^{-\pi|\IM t|}$ up to a power. The absolute $s$ integral is $\Gamma(1+c)$, so Fubini is justified. Integration supplies $\Gamma(1+t)$ and recovers \eqref{eq:quartic-barnes}. Holomorphy then extends the equality to the slit plane.

The auxiliary representative $E(\ee^u;-g)=\exp(-g\ee^u)$ is entire in $u$, whereas the coupling Borel transform $B(\xi)$ has a finite singularity. The ground-state weights change the Borel problem: singularities of the evaluated series cannot be inferred from the additive representative alone.

\paragraph{Exact value and strong coupling.}
The known exact expression provides an independent check \cite{MalbouissonPortugalSvaiter2001,GradshteynRyzhik2014}:
\begin{dmath}\label{eq:quartic-bessel}
 Z_\pi(g)=\frac{\ee^{1/(8g)}}{2\sqrt{\pi g}}
 K_{1/4}\left(\frac1{8g}\right),
 \qquad {|\arg_\pi g|<\pi}
\end{dmath},
with all factors continued from $g>0$. On the covering their logarithms must be continued together, rather than reset to principal values independently.

Leftward contour displacements in \eqref{eq:quartic-barnes} instead cross the state poles $t=-1/4-k/2$. Their residues give the strong-coupling series
\begin{dmath}\label{eq:quartic-strong}
 {Z(g)=g^{-1/4}\sum_{k=0}^\infty(-1)^k C_k g^{-k/2},\qquad
 C_k=\frac{\Gamma(1/4+k/2)}{2\sqrt\pi\,k!}}
\end{dmath}.
Here equality, rather than only an asymptotic relation, is justified by scaling $x=g^{-1/4}y$ for $g>0$. With $h=g^{-1/2}$, expand the quadratic factor in the resulting integral. On bounded sets of complex $h$, the absolute series is dominated by $\exp(-y^4/4+M y^2/2)$ for a suitable $M$, an integrable function. Thus the series after removal of $g^{-1/4}$ is entire in $h$, and continuation gives the identity on the chosen lift. 

Taylor's remainder in the same scaled integral gives the useful two-term bound
\begin{dmath}\label{eq:quartic-strong-remainder}
 {0\leq g^{1/4}Z(g)-(C_0-C_1g^{-1/2})\leq C_2g^{-1},
 \qquad g>0}
\end{dmath}.
In particular, $g^{1/4}Z(g)\to C_0=\Gamma(1/4)/(2\sqrt\pi)$. The two pole families thus encode a divergent expansion at weak coupling and a convergent expansion in the strong-coupling variable.

\subsection{Lateral discontinuity and the nonperturbative scale}\label{sec:quartic-stokes}

Changing the direction of contour displacement is different from comparing two logarithmic determinations. For $a>0$, define the upper and lower bank values by
\begin{dmath}\label{eq:quartic-bank-values}
 {Z_+(-a)=\widetilde Z(a,\pi),\qquad
 Z_-(-a)=\widetilde Z(a,-\pi)}
\end{dmath}.
Both are reached from positive coupling and both lie inside the Barnes convergence sector. They are boundary values of $Z_\pi$, not values of the original real integral, which diverges at $g=-a$. The signs here label the geometrical banks, not the increasing- and decreasing-angle lateral labels of \cref{sec:Laplace_transform}.

Subtracting the two Barnes integrals multiplies their density by $2i\sin(\pi t)$. Reflection therefore gives
\begin{dmath}\label{eq:quartic-jump-barnes}
 Z_+(-a)-Z_-(-a)
 =-2\pi i\,\frac1{2\pi i}\int_{\mathcal C_c}
 a^t\frac{\Gamma(2t+1/2)}{\sqrt\pi\,\Gamma(1+t)}\,\dd t
\end{dmath}.
Subtraction is justified by absolute convergence. The right-hand integral also converges directly: its density decays as $\ee^{-\pi|\IM t|/2}$ up to a power. The perturbative poles of $\Gamma(-t)$ have disappeared. Thus this discontinuity is not the residue sum that produced the weak-coupling expansion.

The coupling Borel cut makes its exact value particularly transparent. For $s>1/4$, the logarithmic continuation of \eqref{eq:quartic-borel} gives
\begin{dmath}\label{eq:quartic-borel-cut}
 B(-s+i0)-B(-s-i0)=-i\sqrt2\,B(s-1/4)
\end{dmath}.
Appendix~\ref{app:quartic-borel-cut} derives the sign and normalization and justifies taking the bank limits under the Laplace integral. The two banks coincide before this threshold. Using \eqref{eq:quartic-borel-sum} and shifting $s$ by $1/(4a)$ now gives the exact connection
\begin{dmath}\label{eq:quartic-exact-stokes}
 Z_+(-a)-Z_-(-a)
 ={-i\sqrt2\,\ee^{-1/(4a)}Z(a),\qquad a>0}
\end{dmath}.
The continuation formula for $K_{1/4}$ supplies an independent check of the same sign and prefactor \cite[Eq.~10.34.2]{DLMF}. Equation~\eqref{eq:quartic-exact-stokes} determines not only the exponentially small scale, but its complete fluctuation function. In particular,
\begin{dmath}\label{eq:quartic-jump-expansion}
 Z_+(-a)-Z_-(-a)
 \sim-i\sqrt2\,\ee^{-1/(4a)}\sum_{n=0}^\infty a_na^n,
 \qquad {a\downarrow0}
\end{dmath},
with a remainder bounded by $\sqrt2\ee^{-1/(4a)}|a_N|a^N$ after $N$ fluctuation terms, by \eqref{eq:quartic-weak-remainder}.

The continued action $x^2/2-a x^4/4$ has saddles $x=\pm a^{-1/2}$ with relative action $1/(4a)$, matching the exponential scale. They are instanton-like saddles, not spacetime instantons. The specified AUF realization thus determines both the algebraic expansion and its lateral connection. A general transseries description would additionally require sector weights and integration-cycle or boundary data.

\section{Conclusion and outlook}\label{sec:conclusion}

We have introduced the Analytic Umbral Framework as an analytic realization of indicial umbral evaluation. Its principal construction pairs the Mellin spectral representation of a Borel representative with an analytic ground state. This defines a continuous action on compatible weighted state spaces. The contour and logarithmic lift are part of the analytic data. The pairing generally converges first in a sector of the logarithmic covering, and its planar determinations are restrictions of the lifted value.

Under the stated hypotheses, finite contour displacement recovers the formal umbral coefficients with an explicit remainder. The resulting estimates distinguish ordinary asymptotics, strong Gevrey asymptotics and equality with a convergent series. Agreement with a Borel--Laplace sum is established through a common uniqueness class or a direct integral identity. The ground state and contour therefore carry information that is not supplied by integer samples alone.

Spectral duality provides a second structural result. On the stated common domains, transferring a Gamma factor between the jump function and the ground state preserves the complete pairing. The exponential--rational correspondence and its admissible P{\'o}lya extensions therefore relate analytic values, not merely formal expansions. The Hankel construction complements the vertical representation for the inverse-Gamma family and establishes entire realizations even when the unchanged vertical integral has no open convergence sector.

The examples range from convergent evaluations to logarithmic residue expansions. The lognormal state gives an ordinary asymptotic expansion beyond every finite Gevrey order. The Gaussian calculation shows how an operational improper integral can be justified using the exact representative. In the quartic model, one pairing yields weak- and strong-coupling expansions, agrees with the coupling Borel sum, and determines the exact lateral connection, including its exponentially small scale and full fluctuation function.

Further work should investigate how differential or difference equations and their boundary data select admissible ground states. A complementary direction is to organize sectorial realizations and their connection laws into a transseries construction with explicit exponential factors. The present results provide a foundation for these developments by specifying the domains and analytic conditions of umbral evaluation.

\appendix

\section{Contour estimates for the analytic pairing}\label{app:pairing-estimates}

\subsection{Uniform control of the residue remainder}\label{app:uniform-recovery}

For the quantitative part of \cref{thm:recovery_formal}, fix $0<\epsilon<a$ and $0<r\leq r_0$, with $|\theta-\pi|\leq a-\epsilon$. On $d_N=N+1/2$, \eqref{eq:uniform-density-bound} implies
\begin{dmath}\label{eq:appendix-displaced-bound}
 |E_N(r,\theta)|
 \leq C_\epsilon B^{N+1/2}
 \Gamma\left(1+s(N+1/2)\right)(N+2)^m r^{N+1/2}
\end{dmath}.
Indeed, $|\ee^{tL}|=r^{d_N}\ee^{-y(\theta-\pi)}$ leaves the integrable factor $\ee^{-\epsilon|y|}$, and
$(1+d_N+|y|)^m\leq (1+d_N)^m(1+|y|)^m$.

The same hypothesis bounds the coefficient without an additional interpolation assumption. Integrate $G$ around $|t-N|=\delta$. The circle lies in the permitted half-plane, avoids the other excluded discs, and has residue $(-1)^{N+1}a_N\varphi(N)$, so
\begin{dmath}\label{eq:appendix-residue-bound}
 |a_N\varphi(N)|
 \leq C_1 B^{N+\delta}
 \Gamma\left(1+s(N+\delta)\right)(N+2)^m
\end{dmath},
after enlarging $C_1$ for the finitely many small indices. For fixed $s>0$ and $d\geq0$, the ratio $\Gamma(1+s(N+d))/\Gamma(1+sN)$ is $O((N+1)^{sd})$ by Stirling's formula. For $s=0$ it equals one. Absorbing these fixed powers of $N+1$ into $D^N$, and using $r^{1/2}\leq r_0^{1/2}$, gives
\begin{dmath}\label{eq:appendix-strong-remainder}
 \left|\widetilde{\mathcal A}(r,\theta)
 -\sum_{n=0}^{N-1}a_n\varphi(n)\zeta^n\right|
 \leq C_2 D^N\Gamma(1+sN)r^N
\end{dmath}.
The last retained coefficient in \eqref{eq:finite-recovery} is included in this remainder. Constants are independent of $N,r,\theta$ in the stated range, including $N=0$. If $s=0$, taking $r<D^{-1}$ makes the remainder tend to zero. At positive order the bound does not assert this fixed-parameter limit.

\subsection{Agreement of the elementary Hankel and vertical realizations}\label{app:hankel-overlap}

Let $\varphi(t)=1/\Gamma(b+\rho t)$, with $b\in\mathbb C$, $\rho>0$, and consider the elementary representatives $E(z;\zeta)=\ee^{\zeta z}$ and $R(z;\zeta)=(1-\zeta z)^{-1}$. Write $v=\theta-\pi$, so that $L=\ln r+iv$. For the rational representative choose $-1<c<0$; for the exponential choose $c<0$. Put $a_R=\pi$ and $a_E=\pi/2$. On the chosen line the elementary jump estimate has the form
\begin{dmath*}
 |\widetilde \jj(t;r,\theta)|
 \leq C r^c(1+|y|)^M
 \ee^{-a_*|y|-yv},\qquad t=c+iy
\end{dmath*},
where $a_*$ is $a_R$ or $a_E$ and $M\geq0$ is fixed. For the rational case this uses the exact cosecant bound; for the exponential it is a Gamma estimate.

Assume $|v|<a_* -\rho\pi/2$. Choose $\pi/2<\chi<\pi$ such that $|v|+\rho\chi<a_*$.
Deform the Hankel tails to $\arg w=\pm\chi$, retaining a counterclockwise circle of sufficiently large radius $R$ between them. Denote this contour by $\mathcal H_{R,\,\chi}$. Throughout it $|\arg w|\leq\chi$, and the tails have exponential decay $\exp(|w|\cos\chi)$. The large radius ensures that any deformation remains outside the substituted profile's finite singularities. The reciprocal-Gamma identity remains valid on this contour.

The lifted auxiliary points $z=w^{-\rho}$ satisfy
$|\arg(-\zeta z)|\leq|v|+\rho\chi<a_*$. Mellin inversion therefore recovers the elementary representative there, continued from its initial ray. Moreover,
\begin{dmath*}
 |\ee^w w^{-b-\rho t}\,\widetilde \jj(t;r,\theta)|
 \leq C_b r^c|\ee^w|\,|w|^{-\RE b-\rho c}
 (1+|y|)^M\ee^{-(a_*-|v|-\rho\chi)|y|}
\end{dmath*}.
The right-hand side is integrable jointly on $\mathcal H_{R,\,\chi}\times\mathbb R$: the circle has positive radius, the tails decay exponentially, and the spectral estimate has a strict margin. Fubini now gives
\begin{dmath*}
 \frac1{2\pi i}\int_{\mathcal C_c}
 \frac{\widetilde \jj(t;r,\theta)}{\Gamma(b+\rho t)}\dd t
 =\frac1{2\pi i}\int_{\mathcal H_{R,\chi}}
 \ee^w w^{-b}\Delta(w^{-\rho};\zeta)\dd w
\end{dmath*}.
Admissible deformation identifies the right-hand side with the fixed-radius Hankel realization. The condition $|v|<a_* -\rho\pi/2$ is precisely \eqref{eq:hankel-overlap-sectors}.

\subsection{The Borel transform obtained from the spectral density}
\label{app:BL-identification}

Under the hypotheses of \cref{prop:AUF-BL-identification}, put
$b=a-\pi s/2>0$ and $H(t)=G(t)/\Gamma(1+st)$. On the fixed line
$t=c_1+iy$, where $\max\{c,-1/s\}<c_1<0$, vertical Stirling estimates
\cite[Section~5.11]{DLMF} give
\begin{dmath*}
 |H(c_1+iy)|\leq C(1+|y|)^M\ee^{-b|y|}
\end{dmath*}
for a suitable $M\geq0$. The same type of estimate holds on every
finite strip to its right, away from the prescribed pole discs.
It proves local uniform convergence of
\eqref{eq:Borel-Barnes-representation} and vanishing of horizontal
connectors during each finite displacement.

A uniform estimate is needed to let the displacement tend to infinity.
For $X\geq1$, Euler's product formula for Gamma
\cite[Section~5.8]{DLMF} gives
\begin{dmath}\label{eq:BL-Gamma-product-bound}
 \frac{\Gamma(X)}{|\Gamma(X+iY)|}
 =\prod_{j=0}^\infty
 \left(1+\frac{Y^2}{(X+j)^2}\right)^{1/2}
 \leq\left(\frac{\sinh(\pi|Y|)}{\pi|Y|}\right)^{1/2}
 \leq\ee^{\pi|Y|/2}
\end{dmath},
with the value at $Y=0$ interpreted by continuity. The middle
inequality follows by replacing $X+j$ with $1+j$ in every factor;
the resulting product is Euler's product for the hyperbolic sine.
Set $X=1+sx$ and $Y=sy$. For $x\geq0$, the uniform density bound
\eqref{eq:uniform-density-bound} consequently implies
\begin{dmath}\label{eq:BL-uniform-normalized-density}
 |H(x+iy)|\leq
 C B^x(1+x+|y|)^m\ee^{-b|y|}
\end{dmath}
outside the pole discs. Reciprocal Gamma introduces no additional poles.

Displace to $d_N=N+1/2$. Since
$\Res_{t=n}H(t)=(-1)^{n+1}b_n/\Gamma(1+sn)$, the clockwise contour
identity gives
\begin{dmath}\label{eq:BL-finite-residue-identity}
 \widetilde B_k(\rho,\vartheta)
 ={
 \sum_{n=0}^{N}\frac{b_n\xi^n}{\Gamma(1+sn)}
 +R_N^B(\rho,\vartheta),\qquad \xi=\rho\ee^{i\vartheta}
 }
\end{dmath}.
For $|\vartheta-\pi|\leq b-\epsilon$, integration of
\eqref{eq:BL-uniform-normalized-density} yields
\begin{dmath}\label{eq:BL-displaced-remainder}
 |R_N^B(\rho,\vartheta)|
 \leq C_\epsilon B^{N+1/2}(N+2)^m\rho^{N+1/2}
\end{dmath}.
This bound tends to zero for $0<\rho<B^{-1}$, uniformly on
compact subsets with a strict angular margin. Hence
\eqref{eq:BL-finite-residue-identity} identifies
$\widetilde B_k$ with the convergent series
\eqref{eq:AUF-coupling-Borel} near zero.

For $\rho\geq1$, the fixed-line estimate gives
$|\widetilde B_k(\rho,\vartheta)|
\leq C_\epsilon\rho^{c_1}$
on proper closed angular subsectors. In the $k$-Laplace
integral, boundedness of the germ controls the origin,
while the exponential factor dominates this algebraic
growth at infinity whenever $|\theta-d|<\pi/(2k)$.
These bounds give local uniform convergence and holomorphy.
Finally, $1+sc_1>0$ justifies Euler's Gamma integral in
\eqref{eq:BL-restores-Gamma}
\cite[Eq.~5.9.1]{DLMF}.

\section{Estimates and contour calculations for the examples}\label{app:example-details}

\subsection{The one-third Gamma denominator}\label{app:third-estimate}

For $x>0$, $v\in\mathbb R$ and $0<\eta<\pi/2$, rotation of Euler's integral to the ray of angle $\eta\operatorname{sgn}v$ gives
\begin{dmath}\label{eq:rotated-gamma-bound}
 |\Gamma(x+iv)|\leq\ee^{-\eta|v|}(\cos\eta)^{-x}\Gamma(x)
\end{dmath}.
The small arc vanishes because $x>0$, and the large arc is controlled by exponential decay. Gamma triplication \cite[Eq.~5.5.6]{DLMF} gives
\begin{dmath*}
 \frac{\Gamma(1+t)}{\Gamma(1+t/3)}
 =\frac{3^{t+1/2}}{2\pi}
 \Gamma\left(\frac{1+t}{3}\right)
 \Gamma\left(\frac{2+t}{3}\right)
\end{dmath*}.
For $t=x+iy$, $x\geq c>-1$, both Gamma arguments have positive real part. Hence the state in \cref{ex:a-third} satisfies
\begin{dmath*}
 |\varphi(x+iy)|
 \leq\varphi(x)(\cos\eta)^{-(2x+3)/3}\ee^{-2\eta|y|/3}
\end{dmath*}.
Real Stirling estimates give $\varphi(x)\leq C D^{x_+}\Gamma(1+2x_+/3)$, where $x_+=\max(x,0)$; the finite interval $c\leq x\leq0$ is absorbed in $C$. Outside fixed discs about the nonnegative integers, $|\csc(\pi(x+iy))|\leq C_\delta\ee^{-\pi|y|}$. Multiplication by the rational jump therefore gives \eqref{eq:uniform-density-bound} with order $s=2/3$ and rate $a=\pi+2\eta/3$. Every closed angular interval strictly inside $|\theta-\pi|<4\pi/3$ is covered by choosing $\eta$ sufficiently close to $\pi/2$. Finally, $\log\varphi(n)=(2/3)n\log n+O(n)$ excludes every smaller coefficient order.

\subsection{Rightward displacement for the Le Roy family}\label{app:leroy-estimate}

Apply \eqref{eq:BL-Gamma-product-bound} with $X=1+d_N$
and $Y=y$, where $d_N=N+1/2$. For $0<\kappa<2$, raising
this estimate to the power $\kappa$ controls the reciprocal
Gamma factor on the displaced contour. Combining it with
$|\csc(\pi(d_N+iy))|\leq2\ee^{-\pi|y|}$ gives, on
$|\theta-\pi|\leq\pi(1-\kappa/2)-\epsilon$, the remainder bound
\begin{dmath}\label{eq:leroy-vanishing-remainder}
 |E_N(r,\theta)|\leq
 C_\epsilon\frac{r^{N+1/2}}{\Gamma(N+3/2)^\kappa}
\end{dmath}.
This tends to zero for every fixed $r$, uniformly on compact parameter sets. Stirling estimates on each finite strip remove the horizontal connectors. At $\kappa=0$, the remaining geometric factor proves the residue identity for $r<1$.

For $\kappa=-p<0$, apply \eqref{eq:rotated-gamma-bound} to $\Gamma(1+t)^p$. Real Stirling bounds give $\Gamma(1+x)^p\leq C D^{x_+}\Gamma(1+px_+)$ for $x\geq c>-1$. Including the cosecant gives the uniform recovery estimate with order $p$ and rate $\pi+p\eta$. Every proper subsector of $|\theta-\pi|<\pi(1+p/2)$ is covered. This controls the actual remainder, while $(n!)^p$ proves the minimal coefficient order.

\subsection{Filtered logarithmic residues}\label{app:third-residues}

At $t=-n+\varepsilon$, $n\geq1$, the numerator of \eqref{eq:third-barnes} without its external power has the expansion
\begin{dmath*}
 \Gamma(-t)\Gamma(1+t)^2
 =\frac1{(n-1)!}
 \left[\frac1{\varepsilon^2}+\frac{\psi(n)}{\varepsilon}+O(1)\right]
\end{dmath*}.
For $3\nmid n$, the derivative of $1/\Gamma(1+t/3)$ contributes $-\psi(1-n/3)/3$, while $\ee^{\varepsilon L}$ contributes $L$. Their combination gives \eqref{eq:third-log-residue}. At $n=3m$,
\begin{dmath*}
 \frac1{\Gamma(1-m+\varepsilon/3)}
 =\frac{(-1)^{m-1}(m-1)!}{3}\varepsilon+O(\varepsilon^2)
\end{dmath*},
so the surviving pole is simple and gives \eqref{eq:third-simple-residue}. These local calculations concern the complete density, not separate assignments of poles to a kernel and state.

\subsection{The Gaussian operational contour}\label{app:gaussian-operational-proof}

Use $0<\arg z<\pi$ on $\mathbb R+ih$. The contour integral in \eqref{eq:gaussian-contour-moments} is entire in $t$: the contour stays away from zero and its Gaussian factor controls all powers uniformly on compact $t$-sets. For $\RE t>-1$, lowering the contour to the upper side of the real axis gives
\begin{dmath*}
 \frac1{2\sqrt\pi}\int_{\mathbb R+ih}\ee^{-z^2/4}z^t\dd z
 =\frac{2^{t-1}}{\sqrt\pi}(1+\ee^{i\pi t})
 \Gamma\left(\frac{1+t}{2}\right)
 =\chi(t)
\end{dmath*}.
Reflection and duplication prove the last equality, and entire continuation extends it to all $t$. The vanishing small indentation requires $\RE t>-1$ only for this intermediate deformation, not for the original contour moment.

For $\zeta>0$, Mellin inversion of $\ee^{i\zeta z}$ on this contour has $\log(-i\zeta)=\ln\zeta-i\pi/2$. Fix $c<0$ and write $t=c+iv$. Its absolute spectral bound is
\begin{dmath*}
 |\Gamma(-t)(-i\zeta z)^t|
 \leq C\zeta^c|z|^c(1+|v|)^M
 \ee^{-\delta(z)|v|},\qquad
 \delta(z)=\frac\pi2-|\arg(-iz)|>0
\end{dmath*},
for a fixed $M\geq0$. On $z=x+ih$, $\delta(z)$ is bounded below by a positive constant times $(1+|x|)^{-1}$. Integrating in $v$ thus costs at most a power of $1+|x|$, still integrable against $|\ee^{-z^2/4}|$. The same statement holds after multiplication by any fixed power of $z$. Fubini identifies this contour evaluation with the exponential AUF pairing, not merely with its formal samples. In particular, reflection gives
\begin{dmath*}
 \Gamma(-t)(-i\zeta)^t\chi(t)
 =\frac12\zeta^t\Gamma(-t/2)
\end{dmath*}.
The resulting Barnes integral equals $\ee^{-\zeta^2}$ by Gamma inversion, initially on $|\arg\zeta|<\pi/4$. Alternatively, lowering the auxiliary contour and completing the square gives the same value. Finally, \eqref{eq:gaussian-joint-integrability} justifies the external integration. The inner Laplace integral is $\Gamma(s)(-iz)^{-s}$, and the proved contour moment supplies $\chi(-s)$ with the phase in \eqref{eq:gaussian-operational-integral}.

\section{The logarithmic Borel cut of the quartic model}\label{app:quartic-borel-cut}

Put $p=1/4$, $q=3/4$ and $H(w)={}_2F_1(p,q;1;w)$, so that $B(\xi)=H(-4\xi)$. Since $p+q=1$, the hypergeometric continuation at $w=1$ has the form
\begin{dmath}\label{eq:quartic-logarithmic-continuation}
 H(w)=D(1-w)-\frac{\log(1-w)}{\Gamma(p)\Gamma(q)}H(1-w)
\end{dmath},
where $D$ is holomorphic near zero \cite[Eq.~15.8.10]{DLMF}. This identity holds near $w=1$ off the cut, with the logarithm continued from $w<1$. For $\xi=-s+i0$, $s>1/4$, the variable $w=4s-i0$ lies below the hypergeometric cut, and $\log(1-w)$ has imaginary part $+\pi$. The lower bank of $\xi$ instead gives imaginary part $-\pi$. The analytic terms cancel in the difference. Since $H(1-4s)=B(s-1/4)$ and $\Gamma(1/4)\Gamma(3/4)=\pi\sqrt2$, this proves
\begin{dmath*}
 B(-s+i0)-B(-s-i0)=-i\sqrt2\,B(s-1/4)
\end{dmath*}
just beyond $s=1/4$. Both sides continue along $s>1/4$, so the identity holds on the entire interval. Below the threshold, $B$ is holomorphic across the negative real axis and the bank difference is zero.

We next justify passing these boundary values through \eqref{eq:quartic-borel-sum}. Fix $a>0$ and approach $g=-a$ along either bank with $|g|=a$. On compact $s$-intervals away from $s_0=1/(4a)$, the continuations are bounded. Near $s_0$, \eqref{eq:quartic-logarithmic-continuation} gives an integrable majorant
\begin{dmath*}
 C\bigl(1+|\log|1-4as||\bigr)
\end{dmath*}.
Indeed, $|1+4gs|\geq|1-4as|$ on these fixed-modulus approaches, while the argument of the logarithm remains bounded. At large $s$, the hypergeometric connection at infinity is a sum of terms of orders $s^{-p}$ and $s^{-q}$, uniformly in a one-sided neighbourhood of each bank; $p-q$ is not an integer \cite[Eq.~15.8.2]{DLMF}. The factor $\ee^{-s}$ controls both tails. Dominated convergence therefore permits the bank limits and their difference under the integral. Translation $s=v+1/(4a)$ gives
\begin{dmath*}
 Z_+(-a)-Z_-(-a)
 =-i\sqrt2\,\ee^{-1/(4a)}
 \int_0^\infty\ee^{-v}B(av)\dd v
\end{dmath*},
and the remaining integral is $Z(a)$ by \eqref{eq:quartic-borel-sum}.

For a check using the exact expression, set $x=1/(8a)$. The reciprocal argument of $K_{1/4}$ approaches angle $-\pi$ on the upper $g$-bank and $+\pi$ on the lower bank. Continuing the square root in \eqref{eq:quartic-bessel} at the same time gives \cite[Eq.~10.34.2]{DLMF}
\begin{dmath}\label{eq:quartic-explicit-banks}
 Z_\pm(-a)=\frac{\ee^{-x}}{2\sqrt{\pi a}}
 \left[\pi I_{1/4}(x)+\frac{1\mp i}{\sqrt2}K_{1/4}(x)\right]
\end{dmath}.
Their difference agrees with the connection formula. In particular, the upper bank has negative imaginary part. Resetting the logarithm of $1/g$ independently of that of $g$ would lose this sign information.

\bibliographystyle{unsrt}
\bibliography{bibliography}
\end{document}